\documentclass[11pt,a4paper]{amsart}

\usepackage{graphicx}
\usepackage{amsmath}
\usepackage{algorithmicx}
\usepackage{algpseudocode}
\usepackage{algorithm}

\usepackage{pgfplots,tikz}
\usetikzlibrary{shapes,arrows,calc}
\usepackage[centering]{geometry}
\usepackage{etoolbox}

\usepackage{multirow}
\usepackage{diagbox}
\usepackage{scalerel,amssymb}
\usepackage{todonotes}
\usepackage[section]{placeins}
\usepackage{listings}
\usepackage{cleveref}
\usepackage{enumerate}
\theoremstyle{plain}
\newtheorem{theorem}{Theorem}[section]
\newtheorem{corollary}[theorem]{Corollary}   
\newtheorem{lemma}[theorem]{Lemma}           
\newtheorem{proposition}[theorem]{Proposition} 

\theoremstyle{definition}
\newtheorem{definition}[theorem]{Definition}

\newtheorem{remark}[theorem]{Remark}

\newcommand{\im}{\operatorname{im}}
\renewcommand{\Re}{\operatorname{Re}}

\newcommand{\rank}{\operatorname{rank}}
\newcommand{\<}{\langle}
\renewcommand{\>}{\rangle}

\newcommand{\dist}{\operatorname{dist}}

\renewcommand{\emptyset}{\varnothing}

\newcommand{\braces}[1]{{\rm (}#1{\rm )}}
\newcommand{\rmref}[1]{{\rm\ref{#1}}}

\newcommand{\R}{\ensuremath{\mathbb R}}    
\newcommand{\C}{\ensuremath{\mathbb C}}    
\newcommand{\N}{\ensuremath{\mathbb N}}    

\newcommand{\gperp}{{[\perp]}}

\newcommand{\product}{[\cdot\,,\cdot]}

\newcommand{\calC}{\mathcal C}         
         
\newcommand{\calE}{\mathcal E}

\newcommand{\calI}{\mathcal I}

\newcommand{\calL}{\mathcal L}

\newcommand{\calR}{\mathcal R}

\newcommand{\calU}{\mathcal U}         
\newcommand{\calV}{\mathcal V}         
         
\newcommand{\calX}{\mathcal X}

\newcommand{\la}{\lambda}
\newcommand{\veps}{\varepsilon}

\newcommand{\vphi}{\varphi}

\newcommand{\mat}[4]
{
	\begin{pmatrix}
		#1 & #2\\
		#3 & #4
	\end{pmatrix}
}

\newcommand{\smallvek}[2]{\left(\begin{smallmatrix}#1\\#2\end{smallmatrix}\right)}
\newcommand{\smat}[4]{\left(\begin{smallmatrix}#1 & #2\\#3 & #4\end{smallmatrix}\right)}

\renewcommand{\Re}{\operatorname{Re}}
\newcommand{\linspan}{\operatorname{span}}

\newcommand{\Lra}{\Longrightarrow}
\newcommand{\Sra}{\Rightarrow}

\newcommand{\wto}{\rightharpoonup}

\newcommand{\ol}{\overline}

\newcommand{\wt}{\widetilde}

\newcommand{\diag}{\operatorname{diag}}
\newcommand{\topp}{\top}
\newcommand{\UU}{\mathbb U}
\newcommand{\VV}{\mathbb V}
\newcommand{\inte}{\operatorname{int}}

\makeatletter
\DeclareRobustCommand{\rvec}[1]{{\mathpalette\rvec@{#1}}}
\newcommand{\rvec@}[2]{%
	\vbox{\offinterlineskip
		\ialign{%
			\hfil##\hfil\cr
			$\m@th#1{}_{\rightharpoonup}$\kern-\scriptspace\cr
			$\m@th#1#2$\cr
		}%
	}%
}
\DeclareRobustCommand{\lvec}[1]{{\mathpalette\lvec@{#1}}}
\newcommand{\lvec@}[2]{%
	\vbox{\offinterlineskip
		\ialign{%
			\hfil##\hfil\cr
			$\m@th#1{}_{\leftharpoonup}$\kern-\scriptspace\cr
			$\m@th#1#2$\cr
		}%
	}%
}
\makeatother

\newcommand{\RF}{\calR}
\newcommand{\RT}{\calC}
\newcommand{\RTV}{\RT_t{}^{\!\!\VV}}
\newcommand{\RFV}{\RF_t{}^{\!\!\VV}}

\begin{document}

\title[Optimal control of port-Hamiltonian descriptor systems]{Optimal control of port-Hamiltonian descriptor systems with minimal energy supply}
\author[T.\ Faulwasser, B.\ Maschke, F.\ Philipp, M.\ Schaller and K.\ Worthmann]{Timm Faulwasser$^{1}$, Bernhard Maschke$^{2}$, Friedrich Philipp$^{3}$, Manuel Schaller$^{3}$ and Karl Worthmann$^{3}$}
	\thanks{}
	\thanks{$^{1}$TU Dortmund University, Institute of Energy Systems, Energy Efficiency and Energy Economics, Germany
		{\tt\small timm.faulwasser@ieee.org}}%
	\thanks{$^{2}$Univ Lyon, Universit{\'e} Claude Bernard Lyon 1, CNRS, LAGEPP UMR 
		5007, France {\tt\small bernhard.maschke@univ-lyon1.fr}}
	\thanks{$^{3}$Technische Universit\"at Ilmemau, Institute for Mathematics, Germany
		{\tt\small \{friedrich.philipp, manuel.schaller, karl.worthmann\}@tu-ilmenau.de}. %
		F.\ Philipp was funded by the Carl Zeiss Foundation within the project \textit{DeepTurb---Deep Learning in und von Turbulenz}. 
			M.\ Schaller was funded by the DFG (project numbers 289034702 and 430154635). %
		K.\ Worthmann gratefully acknowledges funding by the German Research Foundation (DFG; grant WO\ 2056/6-1, project number 406141926). B.~Maschke thanks the Institute of Mathematics of TU~Ilmenau for the invitation in October 2020 which has led to this paper.}
	\maketitle

\begin{abstract}
	We consider the singular optimal control problem of minimizing the energy supply of linear dissipative port-Hamiltonian descriptor systems subject to control and terminal state constraints. To this end, after reducing the problem to an ODE with feed-through term, we derive an input-state turnpike towards a subspace for optimal control of generalized port-Hamiltonian ordinary differential equations. We study the reachability properties of the system and prove that optimal states exhibit a turnpike behavior with respect to the conservative subspace. By means of the port-Hamiltonian structure, we show that, despite control constraints, this turnpike property is global in the initial state. Further, we characterize the class of dissipative Hamiltonian matrices and pencils.
\end{abstract}

\section{Introduction}
Many control systems in application domains as diverse as electrical engineering, mechanics, or thermodynamics can be written as port-Hamilto\-nian systems \cite{Brogliato07,Duindam2009,Jacob2012,van2014port}. In recent years, an implicit description of the energy properties \cite{vdSchaft2018,Schaft2020_DiracLagrangeNonlinear} has led to the following class of linear descriptor systems \cite{beattie2018linear,Mehl2018}
\begin{subequations}\label{eq:phDAE}
	\begin{align}
		\frac{\text{d}}{\text{d}t}{Ex}&= (J-R)Qx + Bu\label{e:phDAE_dyn}, \qquad Ex(0) = w^0\in\im E,\\
		y &= B^\top Qx,\label{e:phDAE_output}
	\end{align}
\end{subequations}
where the skew-symmetric {\em structure matrix} $J$ describes the interconnection structure of the systems, along which the energy flows are exchanged between its parts and preserving the total energy, whereas the symmetric positive semi-definite {\em dissipation matrix} $R$ specifies how the system dissipates energy. The matrix $E$ allows an implicit definition of the energy and when it is singular models algebraic constraints arising from symmetries of the energy. Its product $E^\top Q = Q^\top E\ge 0$ with the matrix $Q$ corresponds to the energy stored in the system. 

For models of physical systems, $\calE(u) = \int_0^Tu(t)^\top y(t)\,\text{d}t$ specifies the {\em energy supplied} to the system over a given time horizon $[0,T]$. In the case of electrical circuits, for instance, the pair of input $u$ and output $y$ variables are equal to the pair of voltage and current at a port of the circuit.

For an overview of further physical examples of $u$ and $y$ for pH systems see \cite[Table B.1, p.~205]{van2014port}.
Besides the supply rate $u^\top y$, also the energy Hamiltonian $H(x)\doteq  \tfrac12 x^\top E^\top Qx$ is key in the analysis of port-Hamiltonian systems as the combination of  both gives the energy balance 
\begin{equation}\label{eq:EnergyBalance}
	H(x(T))-H(x(0)) = \int_0^T u(t)^\top y(t)\,\text{d}t - \int_0^T \|R^{\frac12}Qx(t)\|^2\,\text{d}t.
\end{equation}
From the above equality it can be immediately seen that the dynamics are dissipative with respect to the supply rate $u^\top y$---in the sense of Willems \cite{Willems1972a}---and hence they are passive. Thus stabilization and control  can be approached by passivity-based techniques, cf.\ \cite{Ortega08,van2000l2}.

Despite the widespread interest of pH system for modeling, simulation and control of dynamic systems, surprisingly little has been done in terms of optimal control. In \cite{Schaller2020a} we have shown in the ODE case, i.e., $E=I$ in \eqref{eq:phDAE}, that the inherent dissipativity can be exploited in Optimal Control Problems (OCPs) which require to transfer the system state from a prescribed initial condition at $t=0$ to some target set at $t=T$, while minimizing the supplied energy $\int_0^Tu(t)^\top y(t)\,\text{d}t$. Observe that this bilinear objective implies that the considered OCP is singular, i.e., the Hessian of the OCP Hamiltonian with respect to $u$ is $0$. Hence the analysis of existence of solutions is more difficult and Riccati theory cannot be applied directly. 

In spite of this unfavorable situation, we have shown in \cite{Schaller2020a} that in the special ODE-case with $E = I$ and $Q>0$ the OCP is strictly dissipative at least with respect to a subspace and that optimal controls are completely characterized by the state and its adjoint.\footnote{Specifically, we have proven a sufficient condition for any singular arc to be of order $2$.} First steps towards the infinite-dimensional case are taken in \cite{Philipp2021}. In the present paper, it is our main objective to study the OCP in the more general situation $\det E =0$  with regard to reachability properties of the dynamics, dissipativity, and turnpike behavior, cf.\ Figure~\ref{fig:sb}. Classically, the latter means that for varying initial conditions and horizons the optimal solutions reside close to the optimal steady state for the majority of the time. We refer to \cite{Faulwasser2021} for comments on the historical origins and for a recent overview of turnpike results.

The turnpike phenomenon is usually analyzed in two different situations:
\begin{enumerate}
	\item[(a)] supposing that the OCP is regular allows transcribing the first-order optimality conditions as a system ODEs, cf.\ \cite{Gruene2019, Heiland2020, Pighin2020a, Trelat2015};
	\item[(b)] supposing the underlying system, respectively, the OCP as such is strictly dissipative with respect to a specific steady state, cf.\ \cite{Damm2014, epfl:faulwasser15h, Gruene2016a, Trelat18}.
\end{enumerate}
The present paper refers to neither of these situations. While our approach is strongly based on dissipativity, we will, however, not have to assume this property of the OCP. In contrast, a specific subspace dissipativity   property is inherent to any port-Hamiltonian system.
Indeed our main result  establishes a generalized turnpike behavior of the optimal solutions towards the conservative subspace induced by the nullspace of the dissipation matrix $RQ$, see Theorem \ref{thm:DAE_tp}.

Our approach also has a strong relation to non-linear mechanical systems, where the dynamics and symmetries give rise to a manifold of optimal input-state pairs, cf.\ \cite{Faulwasser2021b, Faulwasser2020a}.  

As is well-known, a linear differential algebraic equation is closely related to its corresponding matrix pencil. In \cite{Mehl2020} the pencils associated to DAEs of the form \eqref{eq:phDAE} are called ``dissipative Hamiltonian''. Here, we adopt this notion and characterize this class of pencils (see Theorem \ref{t:ph_pencil_charac}). We also characterize the regularity of such pencils (Proposition \ref{p:regular}) and when their index is at most one (Corollary~\ref{c:ind1}). We regard these results as interesting contributions in their own right. Yet, since they are somewhat tangential to the main  goal of characterizing the optimal solutions for varying initial conditions and horizon, we outsource them to Appendix \ref{a:matpencils}.

\begin{figure}[!h]
	\centering
	\includegraphics[width=\linewidth]{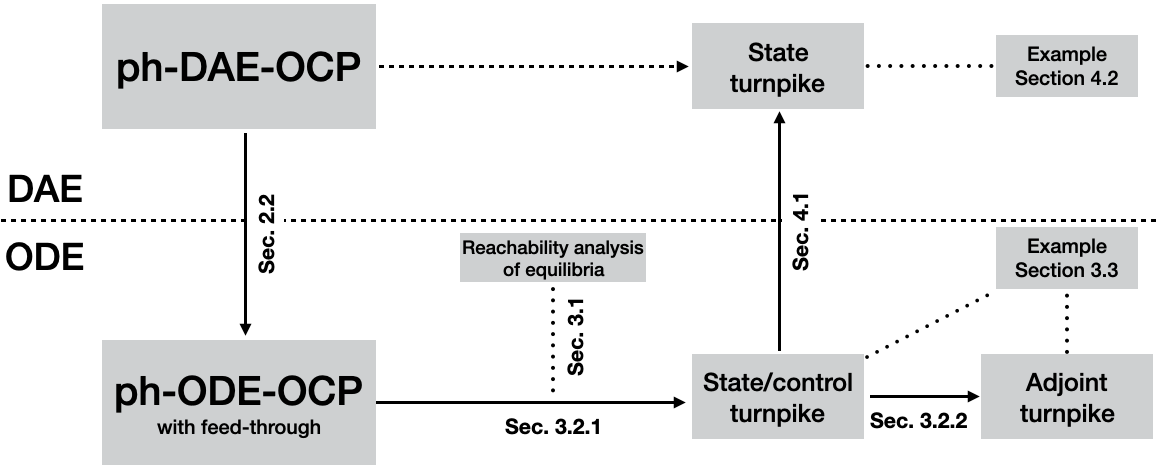}
	\caption{Summary and outline of the main results in this work.}
	\label{fig:sb}
\end{figure} 
The paper is arranged as follows.  Preliminaries and the problem statement are given in Section \ref{s:DAE1}. Therein we introduce the considered OCP for port-Hamiltonian descriptor systems and we show how the algebraic constraints in the DAE can be eliminated via a structure-preserving state transformation proposed in \cite{beattie2018linear}. Having thus reduced the DAE system to a port-Hamiltonian ODE system with feed-through term, Section \ref{s:ODE} studies the corresponding ODE-constrained OCP. In Subsection \ref{ss:ODE_tp} we prove our main result for the ODE case concerning subspace turnpikes. In addition, we prove that the adjoint state performs a turnpike towards zero. The section ends with a numerical example of a modified mass-spring damper system. In Section \ref{s:DAE2} we transfer the statements from Section \ref{s:ODE} to the original DAE case and illustrate  our results by means of a numerical example from robotics. We summarize the paper in Section~\ref{s:conclusion} and give an outlook considering future work.

The paper is supplemented with a comparatively extensive appendix containing a detailed treatise of regular dissipative Hamiltonian pencils, including the above-mentioned characterization.

\
\\
\textbf{Notation. } In the sequel, $\|\cdot\|$ always denotes the Euclidean norm on $\R^n$. The notion $L^1(0,T;\UU)$ stands for the set of all (equivalence classes of) integrable functions on the interval $[0,T]$ with values in $\UU\subset\R^m$. The set of eigenvalues (spectrum) of a real square matrix $A$ is denoted by $\sigma(A)$. The solution of an initial value problem $\dot x = Ax + Bu$, $x(0) = x^0$, with given control $u$ is denoted by $x(\,\cdot\,;x^0,u)$.

\section{Preliminaries}\label{s:DAE1}
In this paper we consider port-Hamiltonian descriptor systems of the form \eqref{eq:phDAE}, where $J,R,Q,E\in\R^{n\times n}$, $B\in\R^{n\times m}$, and
\begin{equation}\label{e:matrix_conditions}
	J=-J^\top,\quad R=R^\top\geq 0,\quad Q^\top E = E^\top Q\geq 0,
\end{equation}
We say that the pair $(u,x)\in L^1_{\rm loc}([0,\infty),\R^m)\times L^1_{\rm loc}([0,\infty),\R^n)$ satisfies \eqref{e:phDAE_dyn}, if $Ex\in W^{1,1}_{\rm loc}([0,\infty),\R^n)$ and \eqref{e:phDAE_dyn} holds almost everywhere on $(0,\infty)$. In this case, we say that $x$ is a solution of \eqref{e:phDAE_dyn}. 

We assume throughout the paper that the DAE \eqref{e:phDAE_dyn} is regular and has differentiation index one (for the definition of both notions see Appendix \ref{a:dae_solutions}), since this guarantees unique solvability of \eqref{e:phDAE_dyn} (cf.\ Proposition~\ref{p:vks} or \cite[Proposition 1]{Ilchmann2019}). Whereas it can be shown that pH-DAEs have a maximal index of two (see Theorem \ref{t:ph_pencil_charac} or \cite{Mehl2018}), many pH-DAEs appear to be of index one, see, e.g., \cite{vanderSchaft2013}. The analysis of OCPs constrained by index two DAEs is subject to future research and can be approached via the index reduction described in \cite[Section 7]{beattie2018linear}.

\begin{remark}[Dissipative Hamiltonian pencils]
	It is well-known that DAEs of the form $\frac{\text{d}}{\text{d}t}Ex = Ax + b$ with $E,A\in\R^{n\times n}$ are closely related to the matrix pencil $P(s) = sE-A$. In Appendix \ref{a:diss-ham-pencil} we coin pencils of the form $P(s) = sE - (J-R)Q$ with matrices $J,R,Q,E\in\R^{n\times n}$ satisfying the conditions~\eqref{e:matrix_conditions} {\em dissipative Hamiltonian} and characterize the regular ones among the class of all regular pencils $P(s) = sE-A$ (see Theorem \ref{t:ph_pencil_charac}). We also characterize regularity of these pencils and the property of having index at most one (see Proposition \ref{p:regular} and Corollary \ref{c:ind1}).
\end{remark}

\subsection{Problem statement}
As already indicated in the Introduction, the power supplied (or withdrawn) from the port-Hamiltonian system \eqref{eq:phDAE} via the ports at time~$t$ is represented by
$
u(t)^\top y(t).
$
Hence, the energy supplied over a time horizon $[0,T]$ is given by
$
\int_0^T u(t)^\top y(t)\,\text{d}t.
$
It is therefore natural to consider the following problem: {\em Given an initial datum $w^0\in\im E$, a target set $\Psi\subset\im E$, and a control set $\mathbb{U}\subset\R^m$, what is the minimal energy supply for steering $w^0$ into $\Psi$?} 

Hence, the considered OCP reads: 
\begin{align}\label{e:phDAE_OCP}
	\begin{split}
		\min_{u\in L^1(0,T;\mathbb{U})} &\int_0^T u(t)^\top y(t)\,\text{d}t\\
		\text{s.t. }\tfrac{\text{d}}{\text{d}t}Ex(t) &= (J-R)Qx(t)+Bu(t),\,\\
		y(t)&=B^\top Qx(t)\\
		Ex(0)&=w^0\in\im E,\, Ex(T)\in\Psi\subset\im E
	\end{split}
\end{align}
We shall assume throughout that the target set $\Psi\subset\im E$ is closed and that the set $\UU\subset\R^m$ of control constraints is convex and compact and contains the origin in its interior $\inte\UU$.
The {\em energy Hamiltonian} of the pH-DAE system \eqref{eq:phDAE} is given by
${H}(x) \doteq   \tfrac12\cdot x^\top E^\top Q x$.
\begin{lemma}\label{l:hamiltonian}
	If $(u,x,y)$ satisfies \eqref{eq:phDAE}, then $H\circ x\in W^{1,1}_{\rm loc}([0,\infty))$ and
	\begin{align}
		\label{eq:energybalance}
		\tfrac{\mathrm{d}}{\mathrm{d}t} H(x(t)) = u(t)^\top y(t) - \|R^{\frac12}Q x(t)\|^2.
	\end{align}
\end{lemma}
\begin{proof}
	Let $P$ denote the orthogonal projection onto $\im E^\top$. Since $I-P$ maps onto $(\im E^\topp)^\perp = \ker E$, we have $E = EP + E(I-P) = EP$. Let $E^\dagger$ denote the Moore-Penrose inverse of $E$. Then $P = E^\dagger E$ and therefore
	$Ex\in W^{1,1}_{\rm loc}([0,\infty),\R^n)$ if and only if 
	$Px\in W^{1,1}_{\rm loc}([0,\infty),\R^n)$.
	Since $H(z) = \frac 12 z^\topp PE^\topp Qz = \frac 12 (Pz)^\topp Q^\topp E(Pz)$, it follows that $H\circ x\in W^{1,1}_{\rm loc}([0,\infty))$ with generalized derivative
	\begin{align*}
		\frac{\text{d}}{\text{d}t}(H\circ x)
		&= x^\topp Q^\topp E\frac{\text{d}}{\text{d}t}Px = x^\topp Q^\topp \frac{\text{d}}{\text{d}t}Ex = y^\topp u - x^\topp Q^\topp RQx,
	\end{align*}
	as $J$ is skew-symmetric and $y = B^\topp Qx$.
\end{proof}

If $(u,x,y)$ satisfies the port-Hamiltonian DAE-system \eqref{eq:phDAE}, then Lemma \ref{l:hamiltonian} immediately implies the following energy balance \eqref{eq:EnergyBalance} which we rewrite as
\begin{align}\label{e:DAE_diss}
	\int_0^T u(t)^\top y(t)\,\text{d}t = H(x(T))-H(x^0) + \int_0^T \|R^{\frac12}Qx(t)\|^2\,\text{d}t.
\end{align}
Put differently, the minimization of the supplied energy is equivalent to minimizing the sum of the overall energy $H(x(T))$ at time $T$ and the internal dissipation given by the last term in \eqref{e:DAE_diss}.

\subsection{Decomposition into differential and algebraic part}
The starting point of our analysis was the OCP \eqref{e:phDAE_OCP} with DAE-con\-straints. In this part, we make use of the structure preserving index reduction from \cite{beattie2018linear} to reduce it to an pH-ODE-constrained OCP. However, the price to pay is a feed-through term in the output equation.
	
	\begin{proposition}\label{p:beattie}
		Suppose that the dynamics of the pH-DAE-constrained OCP pH-DAE \eqref{e:phDAE_OCP} are regular and have index one. Then there exist invertible matrices $U,V\in\mathbb{R}^{n\times n}$ with $U^\top\im E = \R^{n_1}\times \{0\}\subset\R^n$, $n_1\le n$, such that the pH-DAE-OCP~\eqref{e:phDAE_OCP} is equivalent to the pH-ODE-OCP with state $z_1\in\R^{n_1}$, initial datum $z_1^0 \doteq U^\top w^0$ and terminal set $\Phi_1\doteq U^\top\Psi$
		\begin{align}\label{e:phDAE_OCPODE}
			\begin{split}
				\min_{u\in L^1(0,T;\mathbb{U})} &\int_0^T u(t)^\top y(t)\,\text{d}t\\
				\text{s.t. }\dot z_1 &= (J_{11}-R_{11})Q_{11}z_1 +(\hat{B}-\hat{P})u,\quad z_1(0)=z_1^0,\quad z_1(T)\in\Phi_1,\\
				y &= \left(\hat{B}+\hat{P}\right)^\top Q_{11}z_1 + \left(\hat{S}+\hat{N}\right)u,
			\end{split}
		\end{align}
		and 
		\begin{equation}\label{e:z2}
			z_2 = -Q^{-1}_{22}L_{22}^{-1}\big(L_{21}Q_{11}z_1 + B_2u\big)
		\end{equation}
		and corresponding Hamiltonian $\hat{H}(z_1)\doteq  \frac 12\cdot z_1^\top Q_{11}z_1$ satisfying $\hat H(z_1) = H(x)$ and the energy balance
		\begin{align}
			\label{eq:costfunceq_reform}
			\frac{\text{d}}{\text{d}t} \hat{H}(z_1(t)) = y(t)^\top u(t) - \begin{pmatrix}
				z_1(t)\\u(t)
			\end{pmatrix}^\top
			\underbrace{\begin{pmatrix}
					Q_{11} R_{11}Q_{11} & Q_{11} \hat{P}\\
					\hat{P}^\top Q_{11}& \hat{S}
			\end{pmatrix}}_{=:\hat W}
			\begin{pmatrix}
				z_1(t)\\u(t)
			\end{pmatrix}.
		\end{align}
		The transformations $U$ and $V$ satisfy
		\begin{align*}
			\begin{split}
				U^\top EV = \left(\begin{smallmatrix}
					I_{n_1}&0\\
					0&0
				\end{smallmatrix}\right), U^{-1}QV = \left(\begin{smallmatrix}
					Q_{11}&0\\
					0&Q_{22}
				\end{smallmatrix}\right),
				U^\top JU = \left(\begin{smallmatrix}
					{J}_{11}&{J}_{12}\\
					-{J}_{12}^\top&{J}_{22}
				\end{smallmatrix}\right),
				U^\top RU = \left(\begin{smallmatrix}
					{R}_{11}&{R}_{12}\\
					{R}_{12}^\top&{R}_{22}
				\end{smallmatrix}\right)
			\end{split}
		\end{align*}
		with $R_{12} = J_{12}$, $J_{22}-R_{22}$, $Q_{22}$ invertible, $Q_{11}=Q_{11}^\top\ge 0$, $U^\top B = \left(\begin{smallmatrix}
		B_1\\B_2
		\end{smallmatrix}\right)$. Further, abbreviating $L_{ij}\doteq  J_{ij}-R_{ij}$, $i,j=1,2$, we have $\hat{B}= B_1 - \frac12 L_{21}^\top L_{22}^{-\top}B_2$, \mbox{$\hat{P}= -\frac12 L_{21}^\top L_{22}^{-\top}B_2$}, and the symmetric resp.\ skew-symmetric parts of the feed-through operator are given by $\hat{S}= - \frac12 B_2^\top\left(L_{22}^{-1} + L_{22}^{-\top}\right)B_2$ and $\hat{N}= - \frac12 B_2^\top\left(L_{22}^{-1} - L_{22}^{-\top}\right)B_2$.
	\end{proposition}
	\begin{proof}
		The transformations $U$ and $V$ are defined in \cite[Section 5]{beattie2018linear}. Here, we add an additional state transformation $z_1 = E_{11}x_1$, where $E_{11}$ and $x_1$ are as in \cite[Section 5]{beattie2018linear} and transform the dynamics of \eqref{e:phDAE_OCP} via $\left(\begin{smallmatrix}z_1\\z_2\end{smallmatrix}\right) = V^{-1}x$ into
		\begin{subequations}
			\label{e:beattie}
			\begin{align}
				\label{e:beattie_dyn}
				\frac{\mathrm{d}}{\mathrm{d}t}
				\begin{pmatrix}
					I_{n_1}&0\\
					0&0
				\end{pmatrix}
				\begin{pmatrix}
					z_1\\z_2
				\end{pmatrix}
				&=
				\begin{pmatrix}
					\left(J_{11}-R_{11}\right)Q_{11}&0\\
					\left(J_{21}-R_{21}\right)Q_{11}&\left(J_{22}-R_{22}\right)Q_{22}
				\end{pmatrix}
				\begin{pmatrix}
					z_1\\ z_2
				\end{pmatrix}
				+ \begin{pmatrix}
					B_1\\B_2
				\end{pmatrix}u,\\
				\label{e:beattie_obs}
				y&=\begin{pmatrix}B_1^\top,B_2^\top\end{pmatrix}\begin{pmatrix}
					Q_{11}&0\\
					0&Q_{22}
				\end{pmatrix}
				\begin{pmatrix}
					z_1\\z_2
				\end{pmatrix},
			\end{align}
		\end{subequations}
		Further, as in \cite[Theorem 22]{beattie2018linear}, this decomposition and the invertibility of $(J_{22}-R_{22})Q_{22}$ yield the control system
		\begin{align}
			\begin{split}\label{e:phDAE_ODE}
				\dot z_1 = (J_{11}-R_{11})Q_{11} z_1 +(\hat{B}-\hat{P})u,\,
				\quad y = \left(\hat{B}+\hat{P}\right)^\top Q_{11}z_1 + \left(\hat{S}+\hat{N}\right)u,
			\end{split}
		\end{align}
		where $z_2$ is eliminated and uniquely given by \eqref{e:z2}. For the energy balance, see \cite[Proof of Theorem 22]{beattie2018linear}.
	\end{proof}

	\begin{remark}\label{r:xz}
		It follows from Proposition \ref{p:beattie} and Lemma \ref{l:hamiltonian} that, given a control $u\in L^1_{\rm loc}([0,\infty),\R^m)$, we have $x(t)^\top Q^\top RQx(t) = \smallvek{z_1(t)}{u(t)}^\top\hat W\smallvek{z_1(t)}{u(t)}$
		for all solutions $x$ of the DAE in \eqref{e:phDAE_OCP} and the corresponding  solutions $z_1$ of the ODE in \eqref{e:phDAE_ODE}. But in fact, an easy computation shows that
		\begin{align}\label{e:RW}
			\bar x^\top Q^\top RQ\bar x = \smallvek{\bar z_1}{\bar u}^\top\hat W\smallvek{\bar z_1}{\bar u}
		\end{align}
		holds for $\bar x,\bar z\in\R^n$ and $\bar u\in\R^m$ whenever $\bar z=V^{-1}\bar x$ and $\bar z_1,\bar z_2,\bar u$ are related via $\bar z_2 = -Q^{-1}_{22}L_{22}^{-1}\big(L_{21}Q_{11}\bar z_1 + B_2\bar u\big)$. In particular, $\hat W$ is positive semi-definite.
	\end{remark}
	\begin{remark}[Consequences of ODE-reformulation]\label{r:prelim_DAE}
		
		{\bf (a)} Once an OCP is feasible (i.e., the admissible set is non-empty), it is advisable to show the existence of optimal solutions. Since the OCPs \eqref{e:phDAE_OCP} and \eqref{e:phDAE_OCPODE} are equivalent, it suffices to consider the ODE-constrained OCP \eqref{e:phDAE_OCPODE}. And in fact, we prove in Corollary \ref{c:ODE_optsol} below that an optimal solution of OCP \eqref{e:phDAE_OCPODE} exists, whenever the state $z_1^0$ can be steered to a point in $\Phi_1$ at time $T$ under the given dynamics.
		
		{\bf (b)} In view of the energy balance  \eqref{e:DAE_diss} and $H(x) = \frac 12(Ex)^\top QE^\dagger(Ex)$, in the case $RQ=0$ the OCP \eqref{e:phDAE_OCP} is equivalent to minimizing the function $f(w) = w^\top QE^\dagger w$ on the set of states $w\in\Psi$ that are reachable from $w^0$ (inside $\im E$) at time $T$. We are not going to detail this case here. That is, we always assume that $RQ\neq 0$, which means that the system dissipates energy at certain states. Note that this is equivalent to $\hat W\neq 0$ by \eqref{e:RW}.
\end{remark}

\section{Optimal control of pH ODE systems with feed-through}\label{s:ODE}
Next, we  analyze the ODE-constrained OCP \eqref{e:phDAE_OCPODE} with regard to reachability aspects and the turnpike phenomenon showing that optimal solutions stay close to the conservative subspace $\ker\hat W$ for the majority of the time. These properties will then be translated back to the corresponding properties of the original DAE-constrained OCP \eqref{e:phDAE_OCP} in Section~\ref{s:DAE2}.

Throughout this section,  we consider OCPs with dynamic constraints given by port-Hamilto\-ni\-an ODE systems with feed-through
\begin{subequations}\label{e:phODE}
	\begin{align}
		\dot x &= (J-R)Qx + (B-P)u\label{e:phODE_dyn}, \qquad x(0) = x^0,\\
		y &= (B+P)^\top Qx + Du\label{e:phODE_out}.
	\end{align}
\end{subequations}
Moreover, $J,R,Q\in\R^{n\times n}$ and $D\in\R^{m\times m}$ are constant matrices satisfying
\begin{equation}\label{e:phODE_matrix_conditions}
	J=-J^\top,\quad R=R^\top\ge 0,\quad Q=Q^\top\ge 0,\quad S\doteq\tfrac 12(D+D^\top)\ge 0
\end{equation}
and $B,P\in\R^{n\times m}$ such that
\begin{align}\label{e:W}
	W \doteq   \begin{pmatrix}QRQ & QP\\
		P^\top Q & S\end{pmatrix}
	\ge 0.
\end{align}
The latter condition is obviously trivially satisfied if $P=0$.

In \cite{Schaller2020a} we made first observations concerning the special case with $P=0$, $D=0$, and $Q>0$. Subsequently, we go beyond~\cite{Schaller2020a}, i.e., we explore reachability properties of \eqref{e:phODE} and prove novel results concerning input-state subspace turnpikes as well as an adjoint turnpike for the ODE-constrained OCP~\eqref{e:phDAE_OCPODE}.

\begin{remark}[Dissipative Hamiltonian matrices]
	In Appendix \ref{a:diss-ham-mat} we characterize the matrices $A$ which can be written in the form $A = (J-R)Q$ (see Theorem \ref{t:ph-charac}). It turns out that this property is purely spectral, i.e., it can be read off the Jordan canonical form of $A$.
\end{remark}

We say that a linear map $T : \calL\to\calL$, mapping a subspace $\calL\subset\R^n$ (or $\C^n$) to itself is {\em $Q$-symmetric} ({\em $Q$-skew-symmetric}, {\em $Q$-positive semi-definite}), if it has the respective property with respect to the positive semi-definite inner product $[\cdot,\cdot] \doteq   \<Q\cdot,\cdot\>$, restricted to $\calL$. For example, $T$ is $Q$-skew-symmetric if $[Tx,y] = -[x,Ty] \,\, \forall x,y\in\calL$.

Let $A \doteq   (J-R)Q$, where $J,R,Q$ are as in \eqref{e:phODE_matrix_conditions}. By Theorem \ref{t:ph-charac} none of the eigenvalues of $A$ has a positive real part. It follows from the real Jordan form that there exists a (spectral) decomposition
\begin{equation}\label{e:decompN1}
	\R^n = N_1\oplus N_2
\end{equation}
such that both subspaces $N_1$ and $N_2$ are $A$-invariant, $\sigma(A|_{N_1})\subset\rm{i}\R$, and $A|_{N_2}$ is Hurwitz. In particular, we have $A = A_1\oplus A_2$ with respect to the decomposition \eqref{e:decompN1}. The next proposition provides some geometrical insight which also explains the behavior of optimal solutions, see Subsection \ref{ss:example}.

\begin{proposition}\label{p:decomp}
	Let $J,R,Q$ be as in \eqref{e:phODE_matrix_conditions}. Then the decomposition \eqref{e:decompN1} is $Q$-orthogonal, i.e., $\R^n = N_1\oplus_Q N_2$, and $
	\ker Q\,\subset\,N_1\,\subset\,\ker(RQ).$
	Moreover, the representation of $(J-R)Q$ with respect to the decomposition \eqref{e:decompN1} has the form
	\begin{equation}\label{e:cons_diss_dec}
		(J-R)Q = \mat{J_1}00{J_2-R_2},
	\end{equation}
	where $J_1$ and $J_2$ are $Q$-skew-symmetric in $N_1$ and $N_2$, respectively, $R_2$ is $Q$-positive semi-definite, and $J_2-R_2$ is Hurwitz. The eigenvalues of both $J_1$ and $J_2$ are purely imaginary.
\end{proposition}
\begin{proof}
	Let $A \doteq   (J-R)Q$. By definition of $N_1$, we have $\ker Q\subset\ker A\subset N_1$. Denote by $\calL_\la(A)$ the complex algebraic eigenspace of $A$ at $\la\in\C$, i.e., $\calL_\la(A) = \bigcup_{k=0}^n\ker\big((A-\la)^k\big)\,\subset\,\C^n.$
	For a set $\Delta\subset\C$ we will also use the notation $\calL_\Delta(A) \doteq   \linspan\{\calL_\la(A) : \la\in\Delta\}\,\subset\,\C^n$. Denote by $\gperp$ the orthogonality relation with respect to $\product$. We shall now prove that
	\begin{equation}\label{e:orthogonal}
		\la\neq\mu,\;\calL_\mu(A)\subset\ker(RQ)\quad\Lra\quad\calL_\la(A)\;[\perp]\;\calL_\mu(A).
	\end{equation}
	To see this, let us assume that we have already shown that $\ker((A-\la)^k)\,[\perp]\,\calL_\mu(A)$ for some $k\in\N_0$. Let $(A-\la)^{k+1}y=0$ and set $x = (A-\la)y$. Then, by assumption, $x\in(\calL_\mu(A))^\gperp$. Let us furthermore assume that we have already proved that $y\,\gperp\,\ker(A-\mu)^j$ for some $j\in\N_0$. Let $z\in\ker(A-\mu)^{j+1}$ and set $w \doteq   (A-\mu)z$. Then $[y,w]=0$ and thus (as $RQz=0$),
	\begin{align*}
		\ol\la[z,y]
		&= [z,\la y] = [z,Ay-x] = [z,Ay] = \<Qz,(J-R)Qy\> = -\<(J+R)Qz,Qy\>\\
		&= -\<(J-R)Qz,Qy\> = -[Az,y] = -[\mu z+w,y] = -\mu[z,y],
	\end{align*}
	hence $(\ol\la+\mu)[z,y] = 0$. Note that $\ol\la+\mu = 0$ implies that $(\Re\la)(\Re\mu)<0$ (in which case one of $\calL_\la(A)$ and $\calL_\mu(A)$ is trivial) or $\la=\mu$, which we excluded. Hence, $[z,y]=0$ and \eqref{e:orthogonal} is proved.
	
	We set $N_1' \doteq   \calL_{\rm{i}\R}(A)$ and $N_2' \doteq   \calL_{\C^-}(A)$, where $\C^-\doteq  \{z\in\C : \Re z < 0\}$. Obviously, we have $N_1 = N_1'\cap\R^n$ and $N_2= N_2'\cap\R^n$. Since $N_1'\subset\ker(RQ)$ (see \eqref{e:ialpha} and \eqref{e:atzero}), Equation~\eqref{e:orthogonal} shows that $N_1'\,\gperp\,N_2'$ and thus the $Q$-orthogonality of $N_1$ and $N_2$. Hence, $A$ decomposes as in \eqref{e:cons_diss_dec}, where $J_1 = JQ|_{N_1}$, $J_2 = P_2JQ|_{N_2}$, and $R_2 = P_2RQ|_{N_2}$ with $P_2$ denoting the projection onto $N_2$ along $N_1$. It is easy to see that $J_1$ is $Q$-skew-symmetric and that $P_2$ is $Q$-symmetric. The latter implies that $J_2$ is $Q$-skew-symmetric and $R_2$ is $Q$-symmetric and $Q$-non-negative. By construction, we have $\sigma(J_1)\subset\rm{i}\R$. Finally, it follows from $\ker Q\subset N_1$ that $\product$ is positive definite on $N_2$, and hence also $\sigma(J_2)\subset\rm{i}\R$.
\end{proof}

\begin{remark}
	Note that $\ker R_2$ might still be non-trivial. A simple example is given by $Q = I_2$, $J = \smat 01{-1}0$, and $R = \smat 2000$, in which case $N_1=\{0\}$ and therefore $\ker R_2 =\ker R\neq\{0\}$.
\end{remark}

\subsection{Reachability properties}\label{ss:reachable}
As already mentioned in Section \ref{s:DAE1}, we  consider control constraints $u\in\UU$ where $\UU$ is a convex and compact set  with $0\in\inte\UU$. Let $t>0$ and $x\in\R^n$. We say that $z\in\R^n$ is {\em reachable} from $x$ at time $t$ under the dynamics in \eqref{e:phODE_dyn}, if there exists a feasible $u\in L^1(0,t;\UU)$ such that the corresponding state response satisfies $x(t;x,u) = z$. By $\RF_t(x)$ we denote
the set of all states that are reachable from $x$ at time $t$. Similarly, we denote by $\RT_t(x)$ the set of states from which $x$ is reachable (i.e., which can be controlled/steered to $x$) at time $t$. Clearly, $\RT_t(x)$ equals $\RF_t(x)$ with respect to the dynamics in reverse time $\dot x = -(J-R)Qx - (B-P)u$. Moreover, we set $\RF_0(x) \doteq   \RT_0(x) \doteq   \{x\}$ and
$$\RF(x) \doteq   \bigcup_{t\ge 0}\RF_t(x),
\quad \RT(x) \doteq   \bigcup_{t\ge 0}\RT_t(x), \quad
\RF(\Phi) \doteq   \bigcup_{x\in\Phi}\RF(x),
\quad
\RT(\Phi) \doteq   \bigcup_{x\in\Phi}\RT(x).
$$
Hence, $\RT(\Phi)$ is the set of states that can be steered into $\Phi$.
It is well-known that the sets $\RF_t(x)$ and $\RT_t(x)$ are compact and convex for each $t\ge 0$ and each $x\in\R^n$, see, e.g., \cite[Chapter 2, Theorem 1]{leemarkus}.

\subsubsection{Reachability of steady states}
Recall the Kalman controllability matrix $K(A,B) \doteq (B,AB,\ldots,A^{n-1}B)$ of a linear time-invariant control system $(A,B)$ in $\R^n$, i.e., $\dot{x}=Ax+Bu$.
If $\rank K(A,B) = n$, the linear control system $(A,B)$ is {\em controllable}. 

\begin{lemma}[Reachable sets of input-constrained pH systems]\label{l:reachable_zero}
	Consider the pH-system \eqref{e:phODE} with convex and compact input constraint set $\UU$, $0\in\inte\UU$. Let $\calX \doteq   \im K((J-R)Q,B-P)$. Then the following hold:
	\begin{enumerate}
		\item[{\rm (i)}]   $\RT(0) = \calX$.
		\item[{\rm (ii)}]  $\RF(0)\subset\calX$ is convex and relatively open in $\calX$.
		\item[{\rm (iii)}] If $0 < t_1 < t_2$, then $\RT_{t_1}(0)\subset\inte_{\calX}\RT_{t_2}(0)$,
	\end{enumerate}
	where $\inte_\calX$ denotes the interior with respect to the subspace topology of $\calX$.
\end{lemma}
\begin{proof}
	Set $A \doteq   (J-R)Q$ and $\wt B \doteq   B-P$. From the variation of constants formula it easily follows that $\RF(0)\subset\calX$ and $\RT(0)\subset\calX$. Note that $\calX$ is invariant under $A$ and that $\wt B\bar u\in\calX$ for each $\bar u\in\R^m$. Let $\wt A \doteq   A|_{\calX}$. We consider the system $\dot x = \wt Ax + \wt Bu$ in $\calX$ (hence, also with initial values $x(0)\in\calX$). Then $(\wt A,\wt B)$ is controllable and the same is true for the time-reversed system $(-\wt A,-\wt B)$. Also note that $\Re\la\le 0$ for each eigenvalue $\la$ of $\wt A$ (cf.\ Theorem \ref{t:ph-charac}). The properties (i) and (ii) now follow from \cite[p.\ 45, Theorems 1,2,3,5]{Macki2012}, and (iii) is a consequence of \cite[Corollary 17.1]{hermeslasalle}.
\end{proof}
Let $x_T\in\R^n$ and $x^0\in\RT(x_T)$. Then there exists a minimal time $T(x^0;x_T)$ at which $x_T$ is reachable from $x^0$ (see\footnote{The theorem in \cite{Macki2012} is formulated for $x_T=0$ but the proof also works for $x_T\neq 0$.} \cite[p.\ 60, Theorem 1]{Macki2012}). This defines a {\em minimal time function} $T(\,\cdot\,;x_T) : \RT(x_T)\to [0,\infty)$. 
By $B_\veps(S)$ we denote the open $\veps$-neighborhood of a set $S\subset\R^k$. We also write $B_\veps(x) \doteq   B_\veps(\{x\})$.

\begin{corollary}\label{c:mintime}
	The minimal time function $T(\,\cdot\,;0)\!:\!\RT(0)\!\to \![0,\infty)$ is continuous.
\end{corollary}
\begin{proof}
	We adopt the notation from Lemma \ref{l:reachable_zero} and its proof. Recall that $(\wt A,\wt B)$ is controllable (in $\calX = \RT(0)$). Let $x^0\in\calX$, $\veps > 0$, and set $t = T(x^0;0)$. Then $x^0\in\partial_\calX\RT_t(0)$ (see \cite[Lemma 13.1]{hermeslasalle}). By Lemma \ref{l:reachable_zero} (iii) we have $\RT_{t-\veps}(0)\subset\inte_\calX\RT_t(0)$ and $\RT_t(0)\subset\inte_\calX\RT_{t+\veps}(0)$. Hence, there exists $\delta_1 > 0$ such that $B^\calX_{\delta_1}(x^0)\subset\RT_{t+\veps}(0)$, where $B^\calX_r(x) \doteq   B_r(x)\cap\calX$. On the other hand, there exists $\delta_2>0$ such that $B^\calX_{\delta_2}(x^0)\cap\RT_{t-\veps}(0) = \emptyset$. Indeed, otherwise we had $x^0\in\RT_{t-\veps}(0)$ and thus $x^0\in\inte_\calX\RT_t(0)$, which contradicts $x^0\in\partial_\calX\RT_t(0)$. Hence, choosing $\delta = \min\{\delta_1,\delta_2\}$, we have $B^\calX_\delta(x^0)\subset\RT_{t+\veps}(0)\backslash\RT_{t-\veps}(0)$. For $x\in B^\calX_\delta(x^0)$ this implies $t-\veps\le T(x;0)\le t+\veps$, or, equivalently, $|T(x;0)-T(x^0;0)|\le\veps$.
\end{proof}
A pair $(\bar x,\bar u)\in\R^n\times\UU$ is called a {\em steady state} (or {\em controlled equilibrium}) of the dynamics in \eqref{e:phODE_dyn} if $(J-R)Q\bar x + (B-P)\bar u = 0$. In the following, by $\RTV(x)$ and $\RFV(x)$ we denote the reachable sets for \eqref{e:phODE_dyn} with $L^1$-controls taking on their values in $\VV\subset\R^m$.

\begin{lemma}\label{l:karfreitag}
	Let $(\bar x,\bar u)\in\R^n\times\UU$ be a steady state of \eqref{e:phODE_dyn} and set $\VV \doteq   \UU - \bar u$. Then for $t\ge 0$ we have $\RT_t(\bar x) = \bar x + \RTV(0)$ and $\RF_t(\bar x) = \bar x + \RFV(0)$.
\end{lemma}
\begin{proof}
	We have $x\in\RT_t(\bar x)$ if and only if $x = e^{-tA}\bar x - \int_0^te^{-sA}Bu(s)\,\text{d}s$ with $u\in L^1(0,t;\UU)$. Since $\tfrac {\text{d}}{\text{d}s}e^{-sA}\bar x = e^{-sA}(-A\bar x) = e^{-sA}B\bar u$, and therefore $e^{-tA}\bar x - \bar x = \int_0^te^{-sA}B\bar u\,\text{d}s$, it follows that $x\in\RT_t(\bar x)$ if and only if $x = \bar x - \int_0^te^{-sA}Bv(s)\,\text{d}s$ with $v\in L^1(0,t;\VV)$. This proves the claim for $\RT_t(\bar x)$. The claim for $\RF_t(\bar x)$ is proved similarly.
\end{proof}
If $\bar u\in\inte\UU$, then also $\VV = \UU-\bar u$ is compact, convex, and contains $0$ in its interior. Hence, the results obtained so far immediately imply the next corollary.

\begin{corollary}\label{c:mintime_ss}
	Let $(\bar x,\bar u)\in\R^n\times\inte\UU$ be a steady state of \eqref{e:phODE_dyn}. Then:
	\begin{enumerate}
		\item[{\rm (i)}]  $\RT(\bar x) = \bar x + \im K((J-R)Q,B-P)$.
		\item[{\rm (ii)}] The minimal time function $T(\,\cdot\,;\bar x) : \RT(\bar x)\to [0,\infty)$ is continuous.
	\end{enumerate}
\end{corollary}

\subsubsection{Reachability in view of the decomposition \eqref{e:cons_diss_dec}}
With respect to the decomposition $\R^n = N_1\oplus_Q N_2$ from Proposition 
\ref{p:decomp} the control system \eqref{e:phODE_dyn} takes the form
\begin{subequations}\label{e:dec}
	\begin{alignat}{4}
		&\dot x_1 &&= \phantom{R_2--}J_1x_1 + B_1u,\qquad &&x_1(0) &&= x_1^0,\label{e:dec1}\\
		&\dot x_2 &&= (J_2-R_2)x_2 + B_2u,\qquad &&x_2(0) &&= x_2^0,\label{e:dec2}
	\end{alignat}
\end{subequations}
where $B_j = P_j(B-P)$ with $P_j$ denoting the projection onto $N_j$ with respect to the decomposition $\R^n = N_1\oplus N_2$, $j=1,2$.

\begin{lemma}[{\cite[Cor.\ 3.6.7]{sontag}}]\label{l:sontag}
	If $((J-R)Q,B-P)$ is controllable, then we have $\RF(0) = N_1 \oplus_Q (\RF(0)\cap N_2)$.
\end{lemma}
The next corollary is a simple consequence of the previous results and shows in particular that $N_1$ is reachable from anywhere.

\begin{corollary}\label{t:N1reachable}
	If $((J-R)Q,B-P)$ is controllable, then for every $x^0\in\R^n$ there is a bounded set $K_{x^0}\subset N_2$ such that $
	N_1\oplus_Q (\RF(0)\cap N_2)\,\subset\,\RF(x^0)\,\subset\,N_1\oplus_Q K_{x^0}$.
\end{corollary}
\begin{proof}
	If $x\in\R^n$ can be reached from zero, then it can be reached from any $x^0\in\R^n$ by Lemma \ref{l:reachable_zero} (i). This and Lemma \ref{l:sontag} prove the first inclusion. For the second inclusion, let $x\in\RF(x^0)$, $x = x_1 + x_2$ with $x_j\in N_j$, $j=1,2$. Then there exist a time $t>0$ and a control $u\in L^1(0,t;\mathbb{U})$ such that $x_2 = e^{tA_2}x_2^0 + \int_0^t  e^{(t-s)A_2}B_2u(s)\,\text{d}s$, where $A_2 = J_2-R_2$. As $A_2$ is Hurwitz, there exist $\omega >0$, $M\geq 1$ such that $\|e^{tA_2}\|\le Me^{-\omega t}$. Let $R \doteq   \frac{M}{\omega}\|B_2\|\Big(\max_{v\in \mathbb{U}}\|v\|\Big)$. Then
	\begin{align*}
		\|x_2\|
		&\le Me^{-\omega t}\|x_2^0\| + \int_0^t  Me^{-\omega(t-s)}\|B_2\|\|u(s)\|\,\text{d}s\le M\|x_2^0\| + R,
	\end{align*}
	which proves the second inclusion with $K_{x_0} = B_{M\|x_2^0\| + R}(0)\cap N_2$.
\end{proof}

\subsection{Turnpike properties of minimum energy supply ph-ODE-OCPs}\label{ss:ODE_tp}
In Section \ref{s:DAE1} we formulated the OCP corresponding to the minimization of the energy supply of port-Hamiltonian descriptor systems and reduced this DAE-constrained problem to an ODE-constrained OCP of the following form:
\begin{align}\label{e:phODE_OCP}
	\begin{split}
		\min_{u\in L^1(0,T;\mathbb{U})} &\int_0^T u(t)^\top y(t)\,\text{d}t\\
		\text{s.t. }\dot x &= (J-R)Qx+(B-P)u, \quad x(0)=x^0,\quad x(T)\in\Phi,\\
		y&=(B+P)^\top Qx + Du
	\end{split}
\end{align}
where $B,P\in\R^{n\times m}$, $D\in\R^{m\times m}$, and $J,R,Q\in\R^{n\times n}$ are as in \eqref{e:phODE_matrix_conditions}. The existence of solutions of the OCP \eqref{e:phODE_OCP} follows from the compactness of the control constraint set $\mathbb{U}$, cf.\ Theorem~\ref{t:optsol_exists}.
Next, we analyze the turnpike phenomenon of optimal solutions of this OCP. Classically, this means that optimal trajectories reside close to certain states for the majority of the time~\cite{Faulwasser2021}. Here, we will show that this phenomenon occurs in a more general way, i.e., optimal pairs $(x^\star, u^\star$) reside close to a subspace for the majority of the time. Despite the problem being linear quadratic, the presented approach to show the turnpike for the primal variables, i.e., the input-state pair $(x^*,u^*)$, does not utilize the optimality conditions due to the possible occurrence of singular arcs. However, we prove in addition that a combination of the first-order optimality conditions and the subspace turnpike for the primal variables induces a turnpike for the adjoint state towards the steady state zero.

The Hamiltonian function, representing the energy of the system \eqref{e:phODE} is given by $H(x) = \tfrac 12\cdot x^\top Qx$. For a control $u\in L^1(0,T;\UU)$ and $x^0\in\R^n$ the solution $x = x(\,\cdot\,,x^0,u)$ of \eqref{e:phODE_dyn} with $x(0) = x^0$ obviously satisfies
\begin{align*}
	\tfrac{\text{d}}{\text{d}t}H(x(\cdot))
	&= x^\top Q\dot x = x^\top Q(J-R)Qx + x^\top Q(B+P)u - 2x^\top QPu\\
	&= -x^\top QRQx + u^\top y - u^\top Du - 2x^\top QPu = u^\top y - \big\|W^{1/2}\left(\begin{smallmatrix} x\\u\end{smallmatrix}\right)\big\|^2,
\end{align*}
where we used $z^\top Jz=0$ for $z\in\R^n$.
Thus, we obtain the well-known energy balance 
\begin{equation}\label{e:ODE_diss}
	\int_{t_0}^{t_1}u^\top y\,\text{d}t = H(x(t_1)) - H(x(t_0)) + \int_{t_0}^{t_1}\big\|W^{1/2}\smallvek xu\big\|^2\,\text{d}t.
\end{equation}
In particular, this shows that we may replace the cost functional in \eqref{e:phODE_OCP} by
$$
J(u) = H(x(T)) + \int_0^T\big\|W^{1/2}\left(\begin{smallmatrix}x(t)\\u(t)\end{smallmatrix}\right)\big\|^2\,\text{d}t.
$$
The next corollary follows immediately from the existence result Theorem~\ref{t:optsol_exists}.

\begin{corollary}\label{c:ODE_optsol}
	If $x^0\in\RT(\Phi)$, then the OCP \eqref{e:phODE_OCP} has an optimal solution.
\end{corollary}

As already discussed in Remark \ref{r:prelim_DAE}, in the case $W=0$ the OCP \eqref{e:phODE_OCP} is equivalent to minimizing $H(x) = \frac 12\cdot x^\top Qx$ on $\Phi\cap\RF_T(x^0)$, which is the (compact) set of states $x\in\Phi$ that are reachable from $x^0$ at time $T$. This case will not be discussed here. Hence, we will assume throughout that $W\neq 0$.

\subsubsection{Input-state subspace turnpikes}
The following definition of an integral turnpike with respect to a subspace is an extension of an integral turnpike with respect to a steady state, cf.\ \cite[Definition 2.1]{Gruene2019a} and \cite[Section 1.2]{epfl:faulwasser15h}. Turnpike properties with respect to sets was discussed in \cite{Trelat18} and with respect to manifolds in mechanical systems in \cite{Faulwasser2021b}. In the context of unobservable cost functionals, velocity-turnpikes were considered in \cite{Faulwasser2020a,flasskamp2019symmetry,Pighin2020a} which yield a turnpike behavior towards the unobservable subspace.

\begin{definition}[Integral input-state subspace turnpike property]\label{def:turnpike_state_control}
	Let $\ell\in C^1(\R^{n+m})$, $\vphi\in\C^1(\R^n)$, and let $\Phi\subset\R^n$ be closed. We say that a general OCP with linear dynamics of the form
	\begin{align}
		\begin{split}\label{e:lin_OCP}
			\min_{u\in L^1(0,T;\UU)}\,&\varphi(x(T)) + \int_0^T \ell(x(t),u(t))\,\mathrm{d}t\\
			\text{s.t. }\dot x = &Ax + Bu,\quad x(0)=x^0,\quad x(T)\in\Phi
		\end{split}
	\end{align}
	has the {\em Integral input-state subspace turnpike property} on a set $S_{\rm tp}\subset\RT(\Phi)$ with respect to a subspace $\mathcal{V}\subset\R^n\times\R^m$, if there exist continuous functions $F,T : S_{\rm tp}\to [0,\infty)$ such that for all $x^0\in S_{\rm tp}$ each optimal pair $(x^\star ,u^\star )$ of the OCP \eqref{e:lin_OCP} satisfies
	\begin{align}\label{e:integral_tp}
		\int_0^T\dist^2\big((x^\star (t),u^\star (t)),\calV\big)\,\mathrm{d}t\le F(x^0) \quad \text{for all}\quad T > T(x^0)
	\end{align}
\end{definition}
\noindent\begin{remark}[Link to measure turnpikes] \label{rem:measuretp}
	The main feature of this definition is the implication that 
	for $T$ large enough any optimal input-state pair is close to the subspace $\calV$ for the majority of the time. Indeed, if $x^0\in S_{\rm tp}$ and $\veps>0$, for $T > T(x^0)$ we have
	$$
	\mu\big(\{t\in [0,T] : \dist((x^\star (t),u^\star (t)),\calV) > \veps\}\big)\le\tfrac 1{\veps^2}\!\int_0^T\dist^2((x^\star (t),u^\star (t)),\calV)\,\text{d}t\le\tfrac{F(x^0)}{\veps^2},
	$$
	where $\mu$ denotes the standard Lebesgue measure. This behavior of optimal trajectories is called {\em measure turnpike}, cf.\ e.g.\ \cite[Definition 2]{Faulwasser2021}. Here, compared to the usual definition in the literature, the measure turnpike property is with respect to a subspace and the dependence of the upper bound on $\varepsilon$ can explicitly be specified.
\end{remark}
In this section we shall show as our main result in Theorem \ref{t:turnpike_reachability} that under suitable conditions the input-state subspace turnpike property holds for the OCP \eqref{e:phODE_OCP} with respect to the subspace $\ker W$ (see \eqref{e:W}). To this end, we introduce two technical lemmas that are required in the proof of Theorem \ref{t:turnpike_reachability}. 
A steady state $(\bar x^\star ,\bar u^\star )$ of \eqref{e:phODE_OCP} is called {\em optimal} if it is a solution of the following problem:
\begin{align}\label{e:phODE_ssOCP}
	\min_{(\bar{x},\bar{u})\in \mathbb{R}^n\times\mathbb{U}} \bar{u}^\top\bar{y} \quad	\text{s.t. } 0= (J-R)Q\bar{x} + (B - P)\bar{u},\quad 
	\bar{y}= (B + P)^\top Q\bar{x} + D\bar{u}.
\end{align}

\begin{lemma}\label{l:ssimkern}
	A steady state $(\bar x,\bar u)$ of \eqref{e:phODE_OCP} is optimal if and only if $\left(\begin{smallmatrix}
	\bar{x}\\
	\bar{u}
	\end{smallmatrix}\right)\,\in\,\ker W.$
\end{lemma}
\begin{proof}
	Let $\bar z = \smallvek{\bar x}{\bar u}$ be a steady state of \eqref{e:phODE_OCP} and set $\bar{y} = (B + P)^\top Q\bar{x} + D\bar{u}$. Since $\bar x^\top Q^\top JQx = 0$, we obtain
	\begin{align*}
		\bar y^\top\bar u= \bar x^\top Q(B-P)\bar u + 2\bar x^\top QP\bar u + \bar u^\top S\bar u = \bar x^\top QRQ\bar x + 2\bar x^\top QP\bar u + \bar u^\top S\bar u = \bar z^\top W\bar z.
	\end{align*}
	In particular, on the set of constraints in \eqref{e:phODE_ssOCP} the target function $\bar u^\top\bar y$ is non-negative. And since $(0,0)$ obviously is a steady state, the optimal value of \eqref{e:phODE_ssOCP} is zero. Hence, a steady state $\bar z$ is optimal if and only if $\bar z^\top W\bar z = 0$, i.e., $W\bar z = 0$.
\end{proof}

\begin{remark}
	Lemma \ref{l:ssimkern} shows that the optimal steady states of \eqref{e:phODE_OCP} are exactly those pairs $(\bar x,\bar u)\in\R^n\times\UU$, for which $\smallvek{\bar x}{\bar u}$ lies in the kernel of $$\left(\begin{smallmatrix}
	(J-R)Q & B-P\\
	QRQ & QP\\
	P^\top Q & S
	\end{smallmatrix}\right)\in \mathbb{R}^{(2n+m)\times (n+m)}.$$
	In particular, the vectors in $\ker Q\times\{0\}$ are optimal steady states.
\end{remark}

\begin{lemma}\label{l:easystuff}
	Let $A\in\R^{k\times k}$, $A = A^\top\ge 0$. Then for all $x\in\R^k$ we have
	$$
	\la_{\min}\cdot\dist^2(x,\ker A)\,\le\,x^\top Ax\,\le\,\la_{\max}\cdot\dist^2(x,\ker A),
	$$
	where $\la_{\min}$ \braces{$\la_{\max}$} is the smallest \braces{resp.\ largest} positive eigenvalue of $A$.
\end{lemma}
\begin{proof}
	We have $A = U^\top DU$ with $D = \diag(\la_i)_{i=1}^k\ge 0$ and $U\in\R^{k\times k}$ orthogonal. Note that $\max_i\la_i = \|A\|$ and that $P \doteq   U^\top\diag(\delta_{\la_i>0})_{i=1}^kU$ is the orthogonal projection onto $\im A$. Let $x\in\R^k$ and $v\doteq  Ux$. Then
	$$
	x^\top Ax = \sum_i\la_iv_i^2 = \sum_{i:\la_i>0}\la_iv_i^2\le\|A\|\sum_{i:\la_i>0}v_i^2 = \la_{\max}\cdot\|Px\|^2.
	$$
	Similarly, $x^\top Ax\ge\la_{\min}\|Px\|^2$. The claim now follows from $\dist(x,\ker A) = \|Px\|$.
\end{proof}
The next theorem is the main result of this subsection.

\begin{theorem}[Integral input-state subspace turnpike]\label{t:turnpike_reachability}
	Let $(\bar x,\bar u)\in\R^n\times\inte\UU$ be an optimal steady state such that $\bar x\in\RT(\Phi)$. Then the OCP \eqref{e:phODE_OCP} has the integral input-state subspace turnpike property on $\RT(\bar x)$ with respect to $\ker W$.\\
	If additionally, $((J-R)Q,B-P)$ is controllable, the turnpike property is global in the initial state, i.e., $\RT(\bar x) = \R^n$.
\end{theorem}
\begin{proof}
	Set $A \doteq   (J-R)Q$ and $\wt B \doteq   B-P$. First of all, we shall define some constants. Due to the spectral properties of  $A$ (see Theorem \ref{t:ph-charac}), there is $M>0$ such that $\|e^{tA}\|\le 1+Mt$ for all $t\ge 0$. Set $ u_{\max} \doteq   \max\{\|u\| : u\in\UU\}$. The condition $\bar x\in\RT(\Phi)$ means that there exist a time $T_1>0$ and a control $u_1\in L^1(0,T_1;\UU)$ such that $x(T_1,\bar x,u_1)\in\Phi$.
	
	Now, let $x^0\in\RT(\bar x)$. By Corollary \ref{c:mintime_ss} the minimal time $T_0(x^0)\doteq  T(x^0;\bar x)$ at which $\bar x$ can be reached from $x^0$ depends continuously on $x^0$. Define
	$$
	T(x^0) \doteq   T_0(x^0) + T_1
	\qquad\text{and}\qquad
	F(x^0) \doteq   \la_{\min}^{-1}\cdot(G_1 + G_2 + G_0(x^0)),
	$$
	where $\la_{\min}$ denotes the smallest positive eigenvalue of the matrix $W$ and $G_0 : \RT(\bar x)\to [0,\infty)$ as well as the constants $G_1,G_2\ge 0$ are defined by
	\begin{align*}
		G_0(x^0) &\doteq   \|W\|T_0(x^0)\cdot\left[ (1+MT_0(x^0))^2\big(\|x^0\| + \|\wt B\| T_0(x^0)u_{\max}\big)^2 + u_{\max}^2 \right]\\
		G_1 &\doteq   \|W\|T_1\cdot\left[ (1+MT_1)^2\big(\|\bar x\| + \|\wt B\|u_{\max}\big)^2 + T_1u_{\max}^2 \right]\\
		G_2 &\doteq   \tfrac 12\|Q\|(1+MT_1)^2\cdot\big(\|\bar x\| + \|\wt B\| T_1u_{\max}\big)^2.
	\end{align*}
	We will show that $T$ and $F$ are as in Definition \ref{def:turnpike_state_control}. To this end, let $u_0$ be the time-optimal control that steers $x^0$ to $\bar x$ at time $T_0 \doteq   T(x^0;\bar x)$ and let $T > T(x^0) = T_0 + T_1$. Define a control $u$ by
	\begin{align*}
		u(t) \doteq  \begin{cases}
			u_0(t),   &t\in [0,T_0]\\
			\bar u,  &t\in [T_0,T-T_1]\\
			u_1(t-(T-T_1)),  &t\in [T-T_1,T]
		\end{cases}
	\end{align*}
	and denote the state response trajectory by $x$, i.e.,
	$$
	x(t) = 
	\begin{cases}
	e^{tA}x^0 + \int_0^te^{(t-s)A}\wt Bu_0(s)\,\text{d}s, &t\in [0,T_0]\\
	\bar x, &t\in [T_0,T-T_1]\\
	e^{(t-(T-T_1))A}\bar x + \int_{T-T_1}^te^{(t-s)A}\wt Bu_1(s-(T-T_1))\,\text{d}s, &t\in [T-T_1,T].
	\end{cases}
	$$
	The constant value $\bar x$ on $[T_0,T-T_1]$ is due to the fact that $(\bar x,\bar u)$ is a steady state of \eqref{e:phODE_OCP}. Hence, $x$ is a trajectory from $x(0) = x^0$ to a point $x(T)\in\Phi$ and therefore admissible for the OCP \eqref{e:phODE_OCP}. The output is given by $y \doteq   (B+P)^\top Qx + Du$.
	
	Let $(x^\star ,u^\star )$ be an optimal solution of \eqref{e:phODE_OCP} with $x^\star (0) = x^0$ and denote the corresponding output by $y^\star $. By optimality and the energy balance \eqref{e:ODE_diss}, we obtain
	\begin{align*}
		H(x^\star (T))&-H(x^\star (0)) +\int_0^T\big\|W^{\frac12}\begin{pmatrix}x^\star (t)\\u^\star (t)
		\end{pmatrix}\big\|^2\,\text{d}t
		= \int_0^T u^\star (t)^\top y^\star (t)\,\text{d}t\\
		&\le\int_0^T u(t)^\top y(t)\,\text{d}t = H(x(T))-H(x(0)) +  \int_0^T\big\|W^{\frac12}\begin{pmatrix}x(t)\\u(t)
		\end{pmatrix}\big\|^2\,\text{d}t.
	\end{align*}
	Since $H(x)=\tfrac 12\cdot x^\top Qx\ge 0$ for all $x\in\R^n$ and $x^\star (0)=x(0) = x^0$, we obtain
	\begin{align}\label{eq:ineq1}
		\int_0^T\big\|W^{\frac12}\begin{pmatrix}x^\star (t)\\u^\star (t)
		\end{pmatrix}\big\|^2\,\text{d}t  \leq  H(x(T)) +  \int_0^T\big\|W^{\frac12}\begin{pmatrix}x(t)\\u(t)
		\end{pmatrix}\big\|^2\,\text{d}t.
	\end{align}
	For $t\in [0,T_0]$ we have $\|x(t)\|
	\le (1+MT_0)\big(\|x^0\| + \|\wt B\| T_0u_{\max}\big)$
	which implies
	\begin{align*}
		\int_0^{T_0}\big\|W^{\frac12}\begin{pmatrix}x(t)\\u(t)
		\end{pmatrix}\big\|^2\,\text{d}t
		\le\|W\|\int_0^{T_0}\big(\|x(t)\|^2 + \|u(t)\|^2\big)\,\text{d}t \le G_0(x^0).
	\end{align*}
	Similarly, for $t\in [T-T_1,T]$, from $\|x(t)\|
	\le (1+MT_1)\big(\|\bar x\| + \|\wt B\| T_1u_{\max}\big)$
	we obtain
	\begin{align*}
		\int_{T-T_1}^{T}\big\|W^{\frac12}\begin{pmatrix}x(t)\\u(t)
		\end{pmatrix}\big\|^2\,\text{d}t
		\le \|W\|\int_{T-T_1}^{T}\big(\|x(t)\|^2 + \|u(t)\|^2\big)\,\text{d}t \le G_1
	\end{align*}
	and $H(x(T)) = \tfrac 12\cdot x(T)^\top Qx(T)\le \tfrac 12\|Q\|\|x(T)\|^2 \le G_2$.
	Since $[\bar x,\,\bar u]\in\ker W$, cf.\ Lemma \ref{l:ssimkern}, we also have $\int_{T_0}^{T-T_1}\big\|W^{\frac12}\left(\begin{smallmatrix}x(t)\\u(t)
	\end{smallmatrix}\right)\big\|^2\,\text{d}t = 0
	$
	and therefore
	$$
	H(x(T)) + \int_0^T\big\|W^{\frac12}\begin{pmatrix}x^\star (t)\\u^\star (t)
	\end{pmatrix}\big\|^2\,\text{d}t\,\le\,G_1 + G_2 + G_0(x^0) = \la_{\min}\cdot F(x^0),
	$$
	and the first claim follows from Lemma \ref{l:easystuff}. 
	The second claim follows from the fact that the set of initial values $S_{\rm tp} = \RT(\bar x)$ in Theorem \ref{t:turnpike_reachability} coincides with the affine subspace $\bar x + \im K((J-R)Q,B-P)$, see Corollary \ref{c:mintime_ss} (i) and as, due to controllability $ \im K((J-R)Q,B-P) = \mathbb{R}^n$.
\end{proof}

\begin{remark}
	For $\bar x = \bar u = 0$ in the second claim of Theorem~\ref{t:turnpike_reachability} the assumption $0\in\RT(\Phi)$ can be replaced by the (seemingly) weaker condition $(N_1\oplus(\RF(0)\cap N_2))\cap\Phi\ne\emptyset$, where $N_1$ is the subspace from Proposition \ref{p:decomp}. This follows directly from Lemma \ref{l:sontag}.
\end{remark}

\begin{remark}\label{r:only_state}
	If $P=0$ and $S=0$, we have $\ker W = \ker(R^{\frac 12}Q)\times\R^m$ so that the above-proven turnpike property only provides information about the state. The relation \eqref{e:integral_tp} then reads
	$$
	\int_0^T\dist^2\big(x^\star (t),\ker(RQ)\big)\,\text{d}t\le F(x^0).
	$$
\end{remark}

\subsubsection{Classical turnpike for the adjoint state}
In this part, we will show that despite the input-state pair enjoys a subspace turnpike, the adjoint variable exhibits a turnpike towards the steady state zero whenever control constraints are not active. A central tool is the dissipativity equation \eqref{e:ODE_diss} which allows to reformulate the OCP \eqref{e:phODE_OCP} in equivalent form as follows:
\begin{align}
	\label{eq:phOCP_ode_reform}
	\begin{split}
		\min_{u\in L^1(0,T;\mathbb{U})} & H(x(T)) + \int_0^T \big\|W^{\frac12}\left(\begin{smallmatrix}
			x(t)\\u(t)
		\end{smallmatrix}\right)\big\|^2\,\text{d}t\\
		\dot x &= (J-R)Q x +(B-P)u, \quad x(0)=x^0,\quad x(T)\in \Phi,\\
	\end{split}
\end{align}
where $W$ is as in \eqref{e:W}. In order to conclude a result for the adjoint, we shall utilize the optimality conditions which we derive for the OCP \eqref{e:phODE_OCP} following \cite[Section 4.1.2]{Liberzon2012}. First, we define the (optimal control) Hamiltonian
\begin{align*}
	\mathcal{H}(x,u,\lambda,\lambda_0)\doteq  \lambda^\top\left((J-R)Qx + (B-P)u\right) + \lambda_0\big\|W^{\frac12}\smallvek{x}{u}\big\|^2.
\end{align*}
Let $(x^\star ,u^\star )\in W^{1,1}(0,T;\mathbb{R}^n)\times L^1(0,T;\mathbb{U})$ be an optimal input-state pair for \eqref{eq:phOCP_ode_reform}. Then there is a function $\lambda^\star  \in W^{1,1}(0,T;\mathbb{R}^n)$ and a constant $\lambda_0^\star \leq 0$ satisfying $(\lambda_0^\star ,\lambda^\star (t))\neq 0$ for all $t\in [0,T]$ such that
\begin{subequations}
	\label{e:optcond}
	\begin{align}
		\label{eq:state}
		\dot{x}^\star (t)&=\phantom{-}\mathcal{H}_\lambda(x^\star (t),u^\star (t),\lambda^\star (t),\lambda_0^\star )\\
		\label{eq:adj}
		\dot{\lambda}^\star (t) &= -\mathcal{H}_x(x^\star (t),u^\star (t),\lambda^\star (t),\lambda_0^\star )\\
		\label{eq:grad}
		u^\star (t)&\in \arg \max_{u\in \mathbb{U}} \mathcal{H}(x^\star (t),u,\lambda^\star (t),\lambda_0^\star )
	\end{align}
\end{subequations}
for a.e.\ $t\in [0,T]$.

Here, \eqref{eq:adj} and \eqref{eq:grad} read as
\begin{subequations}
	\begin{align}
		\label{eq:adj2}
		\dot{\lambda}^\star  &= -\left((J-R)Q\right)^\top \lambda^\star  - 2\lambda_0^\star \left(Q RQx^\star  +  Q Pu^\star \right)\\
		u^\star (t)&\in \arg\max_{u\in \mathbb{U}}
		\lambda^\star (t)^\top(B-P)u + \lambda_0(2x^\star (t)^\top QPu + u^\top Su).\label{e:argmax}
	\end{align}
\end{subequations}

The proof of the following lemma is inspired by \cite[Proof of Rem.\ 2.1]{Porretta2013}. A similar argument was also pursued in \cite[Theorem 3.5]{Faulwasser2020} in the context of infinite-dimensional nonlinear systems.

\begin{lemma}\label{lem:lambdaestimate}
	Assume that $((J-R)Q,B-P)$ is controllable and let $(x^\star ,u^\star ,\la_0^\star ,\la^\star )$ satisfy the necessary optimality conditions \eqref{e:optcond}. Then for each $t_c\in (0,T)$ there exists a constant $C(t_c)>0$ such that whenever $u^\star (s)\in\inte\UU$ for a.e.\ $s\in [t-t_c,t]$ for some $t\in [t_c,T]$, then
	\begin{align}\label{e:lambda}
		\|\lambda^\star (t)\|^2 \le C(t_c)\cdot\int_{t-t_c}^t\big\|W\!\smallvek{x^\star (s)}{u^\star (s)}\big\|^2\,\text{d}s.
	\end{align}
	In particular, if $t_c < T/4$ and $u^\star (t)\in\inte\UU$ for a.e.\ $t\in [t_c,T-t_c]$, then
	\begin{align}\label{e:lambda_int}
		\int_{2t_c}^{T-t_c}\|\la^\star (t)\|^2\,\text{d}t
		\,\le\,t_cC(t_c)\cdot\int_{t_c}^{T-t_c}\big\|W\!\smallvek{x^\star (t)}{u^\star (t)}\big\|^2\,\text{d}t.
	\end{align}
\end{lemma}
\begin{proof}
	Set $A \doteq   (J-R)Q$ and $\wt B = B-P$. Since $(A,\wt B)$ is controllable, for each $t>0$ there is $\alpha_t > 0$ such that $\int_0^t\|B^\top e^{sA^\top}x\|^2\,\text{d}s \ge \alpha_t\|x\|^2$ for all $x\in \mathbb{R}^n$,
	see \cite[Thm.\ 4.1.7]{Curtain1995}. Let $t\in [t_c,T]$. After a change of variables, this estimate is equivalent to
	\begin{align}\label{eq:obs_bounded}
		\int_{t-t_c}^{t}\|B^\top e^{(t-s)A^\top}x\|^2\,\text{d}s \geq \alpha_{t_c} \|x\|^2 \quad \forall\,x\in\R^n.
	\end{align}
	Using linearity of the dynamics, we decompose the solution of \eqref{eq:adj2} as $\lambda^\star  =  \lambda_1 +  \lambda_2$, where
	\begin{align*}
		&& \lambda_1'(s) &= -A^\top\lambda_1(s), && \lambda_1(t)=\lambda^\star (t),&&\\
		&&\lambda_2'(s) &= -A^\top\lambda_2(s) - 2\lambda_0^\star \left(Q RQx^*(t) +  Q Pu^*(t)\right),&& \lambda_2(t)=0,&&
	\end{align*}
	and apply the observability estimate \eqref{eq:obs_bounded} to $ \lambda_1(s) = e^{(t-s)A^\top} \lambda^\star (t)$. Hence,
	\begin{align*}
		\alpha_{t_c}\|\lambda^\star (t)\|^2 \le \int_{t-t_c}^t \|\wt B^\top\lambda_1(s)\|^2 \,\text{d}s \le 2\int_{t-t_c}^t \big(\|\wt B^\top\lambda^\star (s)\|^2 +\|\wt B^\top \lambda_2(s)\|^2\big)\,\text{d}s.
	\end{align*}
	Now, since $u^\star (s)\in\inte\UU$ for a.e.\ $s\in [t-t_c,t]$, it follows from \eqref
	{e:argmax} that $\wt B^\top \lambda^\star (t) + 2\lambda_0^\star \left(P^\top Qx^\star (t)+  Su^\star (t)\right) = 0$.
	Hence, we obtain
	$$
	\int_{t-t_c}^t\|\wt B^\top\lambda^\star (s)\|^2\,\text{d}s = 4(\lambda_0^\star )^2\int_{t-t_c}^t\|P^\top Qx^\star (t) + Su^\star (t)\|^2\,\text{d}s.
	$$
	Now, setting $F(\tau) = \lambda_0^\star \left(QRQx^\star (\tau) + QPu^\star (\tau)\right)$, we have
	\begin{align*}
		\|\wt B^\top\la_2(s)\|^2
		&\le 4\|\wt B\|^2\cdot\left\|\int_s^t e^{(\tau - s)A^\top}F(\tau)\,\text{d}\tau\right\|^2\\
		&\le 4\|\wt B\|^2\left(\int_s^t\big\|e^{(\tau-s)A^\top}\big\|^2\,\text{d}\tau\right)\left(\int_s^t\|F(\tau)\|^2\,\text{d}\tau\right).
	\end{align*}
	Due to the spectral properties of $A$ (cf.\ Theorem \ref{t:ph-charac}), we have $\|e^{tA}\|\le 1+Mt$ for some $M>0$ and all $t\ge 0$. Hence, also $\|e^{tA^\top}\| = \|(e^{tA})^\top\| = \|e^{tA}\|\le 1+Mt$. The middle term can thus be estimated as
	\begin{align*}
		\int_s^t\big\|e^{(\tau-s)A^\top}\big\|^2\,\text{d}\tau
		&\le\int_s^t(1+M(\tau-s))^2\,\text{d}\tau \le (1+Mt_c)^2t_c.
	\end{align*}
	Finally, integrating the last term and using Fubini's theorem yields
	\begin{align*}
		\int_{t-t_c}^t\int_s^t\|F(\tau)\|^2\,\text{d}\tau\,\text{d}s
		&= \int_{t-t_c}^t\|F(\tau)\|^2\int_{t-t_c}^\tau\,\text{d}s\,\text{d}\tau\\
		&= \int_{t-t_c}^t(\tau-t+t_c)\|F(\tau)\|^2\,\text{d}\tau\le t_c\int_{t-t_c}^t\|F(s)\|^2\,\text{d}s,
	\end{align*}
	and \eqref{e:lambda} follows with $C(t_c) = \frac{8(\lambda_0^\star )^2}{\alpha_{t_c}}\cdot\max\left\{1,\|B-P\|^2(1+Mt_c)^2t_c^2\right\}$.
	Now, let $u^\star (t)\in\inte\UU$ for a.e.\ $t\in [t_c,T-t_c]$. Then we again apply Fubini's theorem to get
	\begin{align*}
		\int_{2t_c}^{T-t_c}\!\!\!\!\!\!\|\la^\star (t)\|^2\,\text{d}t
		\le C(t_c)\!\int_{2t_c}^{T-t_c}\!\!\!\!\int_{t-t_c}^t\!\!\!\!\!\big\|W\!\smallvek{x^\star (s)}{u^\star (s)}\big\|^2\,\text{d}s\,\text{d}t\le t_cC(t_c)\!\int_{t_c}^{T-t_c}\!\!\!\!\!\!\big\|W\!\smallvek{x^\star (s)}{u^\star (s)}\big\|^2\text{d}s,
	\end{align*}
	which is \eqref{e:lambda_int}.
\end{proof}
The following corollary is a consequence of the second claim of Theorem~\ref{t:turnpike_reachability}, and the estimates for the integral $\int_0^T\big\|W^{\frac12}\!\smallvek{x^\star (t)}{u^\star (t)}\big\|^2\text{d}t$ in the proof of Theorem \ref{t:turnpike_reachability}.

\begin{corollary}\label{cor:tp_adj}
	Assume that $((J-R)Q,B-P)$ is controllable, $0\in\RT(\Phi)$, and let $T>T(x^0)$, where $T(\cdot)$ is the function from Theorem \ref{t:turnpike_reachability}. Let $(x^\star ,u^\star ,\la^\star )$ satisfy the necessary optimality conditions for the OCP \eqref{eq:phOCP_ode_reform} and assume that $t_c\in (0,T/4)$ is such that $u^\star (t)\in\inte\UU$ for a.e.\ $t\in [t_c,T-t_c]$. Then the adjoint state $\la^\star$ exhibits an integral turnpike with respect to zero, i.e., there is a continuous function $G:\mathbb{R}^n \to [0,\infty)$ such that
	\begin{align*}
		\int_{2 t_c}^{T-t_c} \|\lambda^\star (t)\|^2\,\mathrm{d}t \leq G(x^0).
	\end{align*}
\end{corollary}

\begin{remark}
		In Corollary~\ref{cor:tp_adj}, we assumed that the control constraints are inactive for the majority of the time. Leveraging the subspace turnpike and the decomposition \eqref{e:dec} it can be shown under additional assumptions that the measure of the time instances, where the control constraints are inactive, grows linearly in the time horizon $T$.
\end{remark}

\subsection{Numerical Example: Modified mass-spring damper system}\label{ss:example}
We briefly illustrate the findings of this chapter via a numerical example of a mass-spring damper system, with homogeneous damping given by the dissipation matrix $R$, cf.\ \cite[Section V]{Schaller2020a}. Going beyond  \cite{Schaller2020a}, we also illustrate the adjoint turnpike. We consider
\begin{align*}
	J\doteq   \begin{pmatrix} 
		\phantom{-}0 & 0 & \phantom{-}1\\
		\phantom{-}0 & 0 & -1\\
		-1 & 1 & \phantom{-}0
	\end{pmatrix}, 
	\quad R\doteq   \begin{pmatrix} 
		1&1&0\\1&1&0\\0&0&0
	\end{pmatrix}
\end{align*}
and $Q=I$, $P=0$, $D=S=0$.
The input matrix, initial and terminal region are given by $B=\left(\begin{matrix}
1,0,0\end{matrix}\right)^\top,\quad x^0=\left(\begin{matrix}
1,1,1\end{matrix}\right)^\top,\quad \Psi = \{x_T\} = \{\left(\begin{matrix}
-1.2,-0.7,-1
\end{matrix}\right)^\top\}$.
We solve the corresponding OCP \eqref{eq:phOCP_ode_reform} with horizon $T\in \{10,15,20\}$ where we discretize the ODE with a RK4-method and $N\in \{100,150,200\}$ time discretization points. The corresponding optimization problem is then solved by the \textit{fmincon} function in \textit{MATLAB}. The subspace turnpike phenomenon proved in Theorem~\ref{t:turnpike_reachability} can be observed in Figure~\ref{fig:orbits} where the optimal state approaches the subspace 
\begin{align}
	\label{e:kern2}
	\ker(RQ) = \ker R = \{x\in \mathbb{R}^3\,|\,x_1+x_2=0\}.
\end{align}
The spiraling state trajectory (see Figure \ref{fig:orbits}) can be explained as follows: first, the state quickly approaches $\ker(RQ)$, as predicted by Theorem \ref{t:turnpike_reachability} and Remark \ref{r:only_state}. Note that $\ker(RQ) = N_1$ in this example, where $N_1$ is as in decomposition \eqref{e:decompN1}. Hence, the state $x_2$ in \eqref{e:dec} approaches zero. In addition, we observe in the top right of Figure \ref{fig:orbits} that the optimal control $u$ also approaches zero, which implies that $x_1(t)\approx e^{J_1t}x$ locally. The spiraling effect now results from the skew-symmetry of $J_1$. Further, as depicted in the bottom of Figure~\ref{fig:orbits}, we observe the turnpike towards zero of the adjoint state as proven in Corollary~\ref{cor:tp_adj}.

\begin{figure}[!ht]
	\centering
	\includegraphics[width=0.49\linewidth]{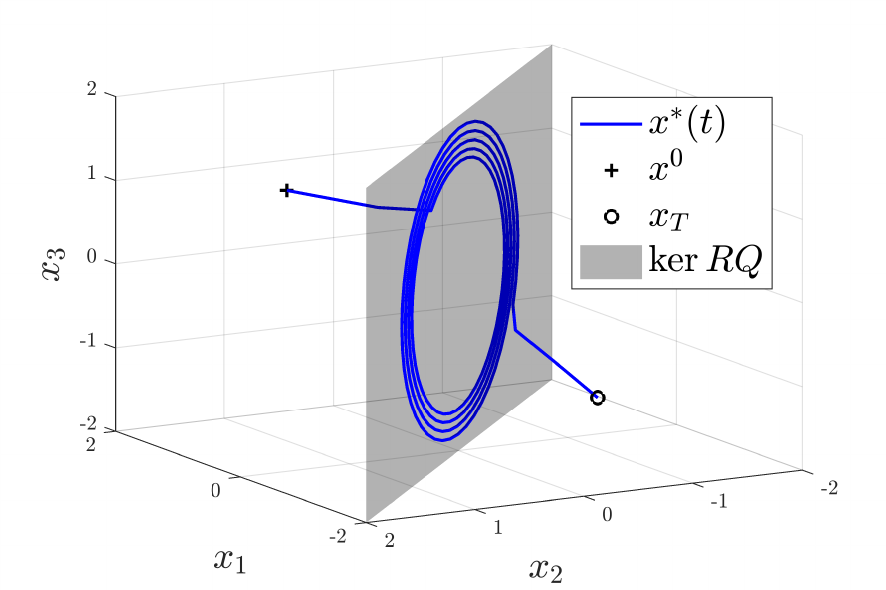} 	\includegraphics[width=0.49\linewidth]{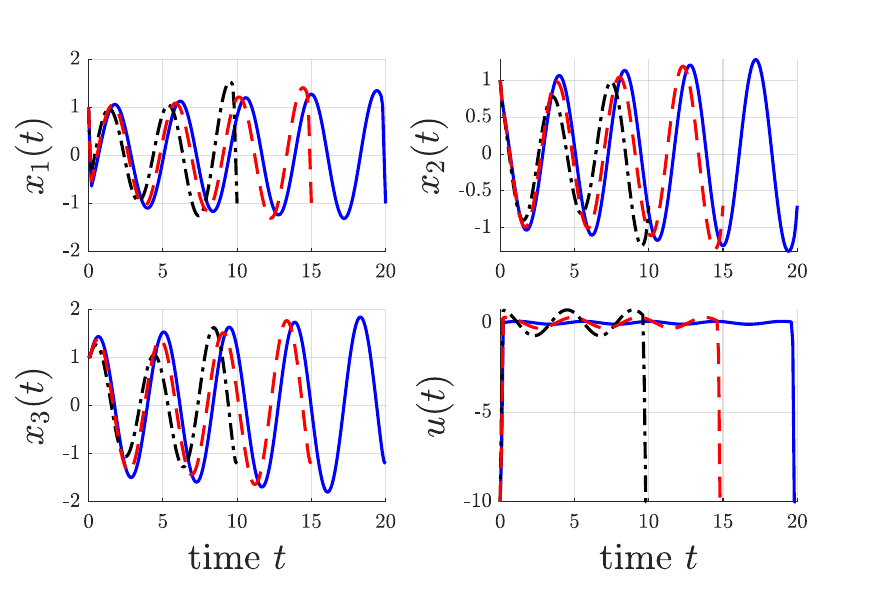}\\
	\includegraphics[width=0.9\linewidth]{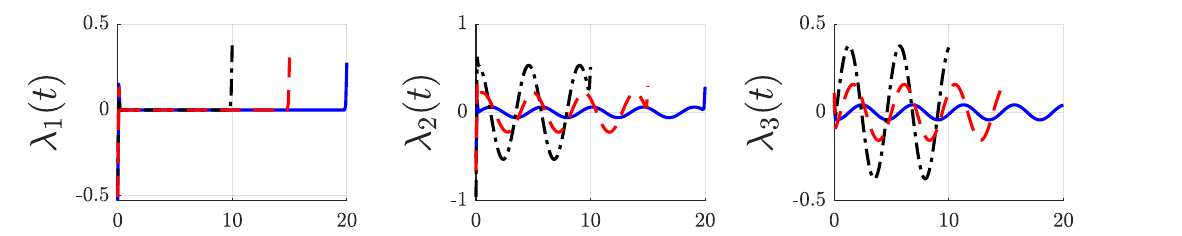}
	\caption{Optimal state and control \braces{top} and corresponding adjoint \braces{bottom} of OCP \eqref{eq:phOCP_ode_reform} for time horizons $T=10$ \braces{\textcolor{black}{\makebox[0.5cm]{\xleaders\hbox to 1.0em{$- \cdot$}\hfill }}}, $T=15$ \braces{\textcolor{red}{\makebox[0.5cm]{\xleaders\hbox to 0.7em{$-$}\hfill }}}, and $T=20$  \braces{\textcolor{blue}{\makebox[0.4cm]{\xleaders\hbox to 1em{---}\hfill }}\label{fig:orbits}}.}
\end{figure}

	\begin{remark}
		In Figure~\ref{fig:orbits}, the optimal trajectory seems to approach a periodic orbit on the turnpike. However, we note that in our setting, the turnpike can not be obtained by means of solving a reduced periodic optimal control problem as in, e.g., \cite[Section 2.2]{trelat2018steady}. Here, a connection between the reduced problem (i.e., the steady state problem \eqref{e:phODE_ssOCP}) and the turnpike subspace is given by Lemma~\ref{l:ssimkern}. However, contrary to classical turnpike results, Lemma~\ref{l:ssimkern} does not yield a reduced OCP that fully characterizes the turnpike set as we only proved that the turnpike subspace contains the set of optimal equilibria, not vice versa. In Figure~\ref{fig:neu}, we show that the seemingly periodic orbit on the turnpike depends on the choice of initial and terminal datum.
		\begin{figure}[!ht]
			\centering
			\includegraphics[width=0.49\linewidth]{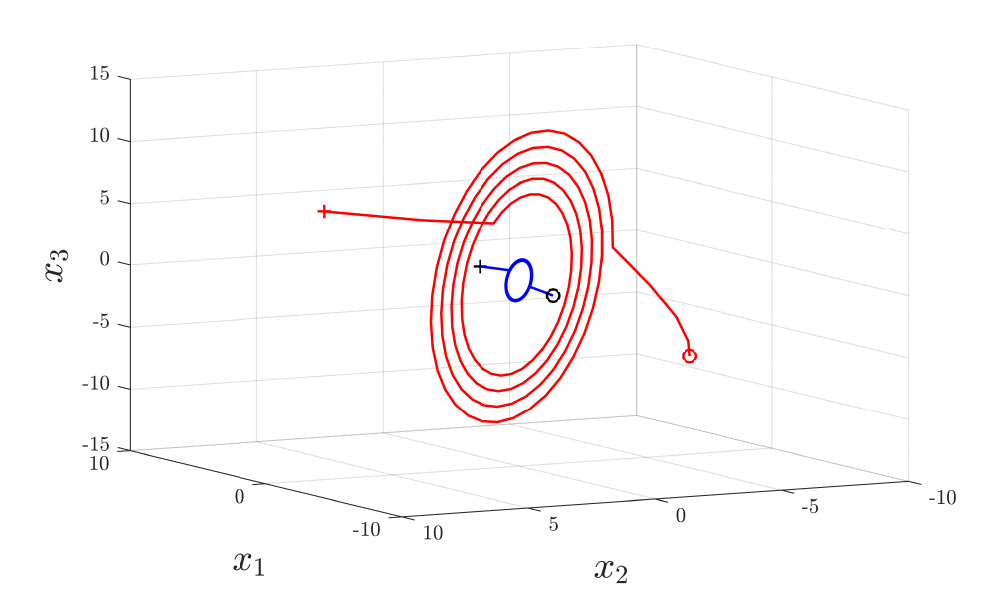}
			\caption{Depiction of two optimal states for different initial and terminal values. The inner trajectory (blue) shows the optimal state of Figure~\ref{fig:orbits}, the outer trajectory (red) the optimal state with the same horizon and with initial and terminal state with each component multiplied by five.}\label{fig:neu}
		\end{figure}
\end{remark}

\section{Port-Hamiltonian DAE-OCPs}\label{s:DAE2}
Subsequently, we leverage the results from  Section \ref{s:ODE} to analyze the pH-DAE OCP \eqref{e:phDAE_OCP}. To this end, let us first discuss reachability properties of the DAE control system~\eqref{e:phDAE_dyn} of index at most one. Let $t>0$ and $w,w^0\in\im E$. We say that $w$ is reachable from $w^0$ at time $t$ under the dynamics in \eqref{e:phDAE_dyn} if there exists a control $u\in L^1(0,t;\UU)$ such that the (possibly non-smooth) solution $x_u$ of the DAE in \eqref{e:phDAE_dyn} with $Ex_u(0) = w^0$ satisfies $Ex(t) = w$. By $\RF_t(w)$ we denote the set of all vectors in $\im E$ that are reachable from $w\in\im E$ at time $t$. Similarly, we denote by $\RT_t(w)$ the set of vectors in $\im E$ from which $w$ is reachable at time $t$. The sets $\RF(w)$ and $\RT(w)$ are defined analogously to their ODE-counterparts in Subsection \ref{ss:reachable} and so are $\RF(\Psi)$ and $\RT(\Psi)$ for sets $\Psi\subset\im E$.
Using the quasi-Weierstra\ss\ form (see \cite{Berger2012}) it is easy to see that the properties of the reach\-able sets for ODEs carry over to the DAE case---with the exception that these sets are contained in $\im E$ and topological properties have to be regarded in the subspace topology of $\im E$.

\subsection{Turnpike properties of minimum energy supply ph-DAE-OCPs}
We shall now define the subspace turnpike property for DAE-OCPs with respect to the state variable, which is the DAE-counterpart to Definition~\ref{def:turnpike_state_control} for ODE problems. Due to the absence of a feed-through term in the ph-DAE-OCP \eqref{e:phDAE_OCP}, we obtain a subspace turnpike purely in the state, as opposed to the input-state turnpike in Definition~\ref{def:turnpike_state_control}.

\begin{definition}[Integral state subspace turnpike property]\label{def:turnpike_state}
	We say that a general DAE-OCP of the form
	\begin{align}
		\begin{split}\label{e:DAE_OCP}
			\min_{u\in L^1(0,T;\UU)} &\varphi(x(T)) + \int_0^T \ell(x(t),u(t))\,\text{d}t \\
			\text{s.t. }\tfrac{\text{d}}{\text{d}t}Ex &= Ax + Bu, \quad Ex(0)=w^0,\quad Ex(T)\in\Psi,
		\end{split}
	\end{align}
	with $C^1$-functions $\ell : \R^{n+m}\to\R$, $\vphi : \R^n\to\R$ and a closed set $\Psi\subset\im E$ has the {\em integral state subspace turnpike property} on a set $S_{\rm tp}\subset\RT(\Psi)$ with respect to a subspace $\mathcal{V}\subset\R^n$, if there are continuous functions $F,T : S_{\rm tp}\to [0,\infty)$ such that for all $w^0\in S_{\rm tp}$ each optimal pair $(x^\star ,u^\star )$ of the OCP \eqref{e:DAE_OCP} satisfies
	\begin{align*}
		\int_0^T\dist^2(x^\star (t),\mathcal{V})\,\text{d}t\,\le\,F(w^0) \quad \text{for all} \quad T > T(x^0).
	\end{align*}
\end{definition}
Let us also define the (optimal) steady states for the DAE-constrained OCP \eqref{e:phDAE_OCP}.

\begin{definition}
	A pair of vectors $(\bar w,\bar u)\in\im E\times\UU$ is called a {\em steady state} of \eqref{e:phDAE_OCP} if there exists $\bar x\in\R^n$ such that $E\bar x = \bar w$ and $(J-R)Q\bar x + B\bar u = 0$. The steady state $(\bar w,\bar u)$ is called {\em optimal} if 
	it is a solution of the following minimization problem:
	\begin{align}\label{e:phDAE_ssOCP}
		\min_{(w,u)\in\im E\times\UU} u^\top y\quad 
		\text{s.t. } 0&= (J-R)Qx + Bu,\quad
		y= B^\top Qx,\quad  w = Ex.
	\end{align}
\end{definition}

A vector $\bar x\in\R^n$ with $E\bar x = \bar w$ and $(J-R)Q\bar x + B\bar u = 0$ is unique. This follows directly from the regularity of the pencil $P(s) = sE-(J-R)Q$.

\begin{lemma}
	$(\bar w,\bar u)\in\im E\times\UU$ is an \braces{optimal\,} steady state of \eqref{e:phDAE_OCP} if and only if $(U^\top\bar w,\bar u)$ is an \braces{optimal\,} steady state of \eqref{e:phDAE_OCPODE}, where $U$ is as in Proposition \rmref{p:beattie}. In particular, a steady state $(\bar w,\bar u)$ of \eqref{e:phDAE_OCP} is optimal if and only if $\bar x\in\ker(RQ)$.
\end{lemma}
\begin{proof}
	We use the notation from Proposition \ref{p:beattie}. Setting $\bar z = V^{-1}\bar x$, the equation $(J-R)Q\bar x + B\bar u=0$ is equivalent to
	\begin{equation}\label{e:malwieder}
		(J_{11}-R_{11})Q_{11}\bar z_1 + B_1\bar u = 0
		\qquad\text{and}\qquad
		\bar z_2 = -Q_{22}^{-1}L_{22}^{-1}(L_{21}Q_{11}\bar z_1 + B_2\bar u).
	\end{equation}
	Now, if $(\bar w,\bar u)$ is a steady state of \eqref{e:phDAE_OCP} and $\bar w = E\bar x$, then $U^\top\bar w = U^\top EV\bar z = \bar z_1$ (see Proposition~\ref{p:beattie}), so that $(U^\top\bar w,\bar u)$ is a steady state of \eqref{e:phDAE_OCPODE}. Conversely, if $(U^\top\bar w,\bar u)$ is a steady state of \eqref{e:phDAE_OCPODE} and we set $\bar z_1 \doteq   U^\top\bar w$, $\bar z_2$ as in \eqref{e:malwieder}, and $\bar x \doteq   V\bar z$, then $E\bar x = U^{-\top}(U^\top EV)\bar z = U^{-\top}\bar z_1 = \bar w$, which shows that $(\bar w,\bar u)$ is a steady state of \eqref{e:phDAE_OCP}. The equivalence of optimal steady states follows from the fact that the transformation in Proposition \ref{p:beattie} does not change the input $u$ and the output $y$. The ``in particular''-part is a consequence of Remark \ref{r:xz} and Lemma \ref{l:ssimkern}.
\end{proof}
The following theorem is our main result concerning the turnpike behavior of optimal solutions of the pH-DAE OCP \eqref{e:phDAE_OCP}.

\begin{theorem}[Integral state subspace turnpikes]\label{thm:DAE_tp}
	Let $(\bar w,\bar u)\in\im E\times\inte\UU$ be an optimal steady state of \eqref{e:phDAE_OCP} such that $\bar w\in\RT(\Psi)$. Then the OCP \eqref{e:phDAE_OCP} has the integral state subspace turnpike property on $\RT(\bar w)$ with respect to $\ker(RQ)$.
\end{theorem}
\begin{proof}
	Let $\bar z_1 \doteq   U^\top\bar w$. Then $(\bar z_1,\bar u)$ is an optimal steady state of OCP \eqref{e:phDAE_OCPODE}. It is easily seen that $\RT(\Psi) = U^{-\top}\RT_{\rm ODE}(\Phi_1)$, where $\RF_{\rm ODE}$ denotes the reachable set with respect to the ODE system in \eqref{e:phDAE_OCPODE}. Therefore, $\bar w\in\RT(\Psi)$ is equivalent to $\bar z_1\in\RT_{\rm ODE}(\Phi_1)$. Hence, by Theorem \ref{t:turnpike_reachability} there exist continuous functions $\wt F,\wt T : \RT_{\rm ODE}(\bar z_1)\to [0,\infty)$ such that for all $z_1^0\in\RT_{\rm ODE}(\bar z_1)$ each optimal pair $(z_1^\star ,u^\star )$ of the OCP \eqref{e:phDAE_OCPODE} with initial datum $z_1^\star (0)=z_1^0$ and $T > \wt T(z_1^0)$ satisfies $\int_0^T\dist^2\big((z_1^\star (t),u^\star (t)),\ker\hat W\big)\,\text{d}t\,\le\,\wt F(z_1^0)$.
	Define $F,T : \RT(\bar w)\to [0,\infty)$ by $F(w) \doteq   \|\hat W\|\la_{\min}^{-1}\cdot\wt F(U^\top w)$ and $T(w) \doteq   \wt T(U^\top w)$, $w\in\RT(\bar w)$, where $\la_{\min}$ is the smallest positive eigenvalue of $Q^\top RQ$. Let $w^0\in\RT(\bar w)$ and let $(x^\star ,u^\star )$ be an optimal pair of \eqref{e:phDAE_OCP} with initial datum $Ex(0) = w^0$. Set $z_1^\star  \doteq   U^\top Ex^\star $ and $z_1^0 \doteq   U^\top w^0$. Then $(z_1^\star ,u^\star )$ is an optimal pair of \eqref{e:phDAE_OCPODE} with $z_1^\star (0) = z_1^0$ and for $T > T(w^0)$ we have $T > \wt T(z_1^0)$ and thus (see Lemma \ref{l:easystuff} and Remark \ref{r:xz}) the claim follows with
	\begin{align*}
		\int_0^T\!\!\dist^2(x^\star (t),\ker(RQ))\,\text{d}t
		&\le \la_{\min}^{-1}\int_0^T\|R^{\frac 12}Qx^\star (t)\|^2\,\text{d}t = \la_{\min}^{-1}\int_0^T\big\|\hat W^{\frac 12}\smallvek{z_1^\star (t)}{u^\star (t)}\big\|^2\,\text{d}t\\
		&\le \frac{\|\hat W\|}{\la_{\min}}\int_0^T\dist^2\big((z_1^\star (t),u^\star (t)),\ker\hat W\big)\,\text{d}t\,\le\,F(w^0).
	\end{align*}
\end{proof}
Recall that a regular DAE control system $\tfrac {\text{d}}{\text{d}t}Ex = Ax + Bu$ in $\R^n$ (or simply $(E,A,B)$) is called {\em R-controllable} if $\rank[\la E - A\quad B] = n$ for all $\la\in\C$. This is obviously a generalization of the Hautus test. 
We have
\begin{align*}
	\rank[\la E - (J-R)Q\quad B]
	&= \rank\big[\la U^\top E - U^\top(J-R)UU^{-1}Q\quad U^\top B\big]\\
	&= \rank\big[\la U^\top EV - U^\top(J-R)UU^{-1}QV\quad U^\top B\big]\\
	&= \rank\left[\begin{pmatrix}\la - L_{11}Q_{11} & 0 & B_1\\-L_{21}Q_{11} & -L_{22}Q_{22} & B_2\end{pmatrix}\right]\\
	&= \rank\big[\la - L_{11}Q_{11}\quad B_1\big] + n_2,
\end{align*}
where $n_2 \doteq   n - n_1$ and $L_{ij} = J_{ij} - R_{ij}$, $i,j=1,2$. The last equality holds since $L_{22}Q_{22}$ is invertible. Hence, $(E,(J-R)Q,B)$ is R-controllable if and only if the ODE control system in \eqref{e:phDAE_OCPODE} is controllable and we obtain directly from the global turnpike result of Theorem~\ref{t:turnpike_reachability}.

\begin{corollary}
	Assume that $(E,(J-R)Q,B)$ is R-controllable. If there exists an optimal steady state $(\bar w,\bar u)\in\im E\times\inte\UU$ of \eqref{e:phDAE_OCP} such that $\bar w\in\RT(\Psi)$, then the OCP \eqref{e:phDAE_OCP} has the state integral turnpike property on $\im E$ with respect to $\ker(RQ)$.
\end{corollary}

\begin{remark}
	The notion of R-controllability for DAE control systems introduced above was first defined in \cite{yip1981}. In \cite{berger2013controllability} the authors show that this property is equivalent to the so-called {\em controllability in the behavioral sense}. However, both \cite{yip1981} and \cite{berger2013controllability} work with DAEs of the type $E\dot x = Ax+Bu$ (instead of $\frac{\text{d}}{\text{d}t}Ex = Ax+Bu$) which are more restrictive due to the regularity requirement on $x$.
\end{remark}

\subsection{Numerical example: Force control of a robot in vertical translation of the end-effector}
Let us consider the force control of a robot manipulator as described in \cite{Volpe1994}. The robot is the type CMU DD II and its end-effector is endowed with a force sensor. We slightly adapt the parameters from \cite{Volpe1994} and set the mass to $m_A = 1.1$, $m_B = 0.1$, and the stiffness parameters to $k_1=0$, $k_2 =5$, and $k_3=\infty$. The choice of the stiffness coefficient  $k_3$ induces clearly a constraint: the elongation of the spring $3$ is taken to $0$, hence yielding a singular matrix $E$(see a similar example in \cite{vdSchaft2018}). The damping parameters are set to $c_1 = 10$, $c_2=10$, and $c_3=17$. The structure and dissipation matrix are given by
\begin{align*}
	R=\left(\begin{smallmatrix}
		0_{3} & 0_{3\times2}\\
		0_{2\times3} & \left(\begin{smallmatrix}
			\left(c_{1}+c_{2}\right) & -c_{2}\\
			-c_{2} & \left(c_{2}+c_{3}\right)
		\end{smallmatrix}\right)
	\end{smallmatrix}\right),\quad  \text{and }\quad J=\left(\begin{smallmatrix}
		0_{3} & \Gamma\\
		-\Gamma^{\top} & 0_{2}
	\end{smallmatrix}\right), \quad \text{where}\quad
	\Gamma=\left(\begin{smallmatrix}
		1 & 0\\
		-1 & 1\\
		0 & -1
	\end{smallmatrix}\right).
\end{align*}
Let further $E=\textrm{diag}(1,\,1,\,\frac{1}{k_{3}},m_{A},m_{B})$, $Q=\textrm{diag}(k_{1},\,k_{2},\,1,\,1,1)$ (with the convention $\tfrac{1}{\infty}=0$), $B=(0,0,0,1,0)^{\top}$, $w^0=(1,1,0,1,0)^\top$, and $\Psi = \{w_T\} = \{(1,1,0,2,0)^\top\}$. We eliminate the algebraic constraints and discretize the corresponding three dimensional ODE for time horizons $T\in \{5,10,15\}$ by discretization with a RK4 method with $N\in \{1000,2000,3000\}$ time steps. The resulting OCP is then solved by \textit{CasADi} \cite{Andersson2019}.
\begin{figure}[!ht]
	\centering
	\includegraphics[width=.9\linewidth]{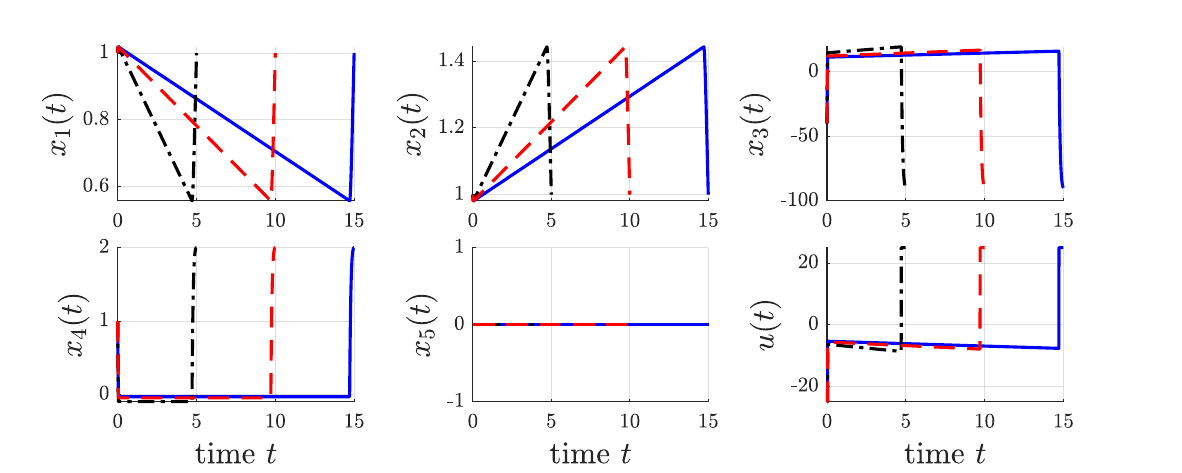}
	\caption{Optimal state and control of the DAE-OCP \eqref{e:phDAE_OCP} for time horizons $T=5$ (\textcolor{black}{\makebox[0.5cm]{\xleaders\hbox to 1.0em{$- \cdot$}\hfill }}), $T=10$ (\textcolor{red}{\makebox[0.5cm]{\xleaders\hbox to 0.7em{$-$}\hfill }}), $T=15$  (\textcolor{blue}{\makebox[0.4cm]{\xleaders\hbox to 1em{---}\hfill }}).\label{fig:dae}}
\end{figure}
Here we have $\ker(RQ) = \{x\in \mathbb{R}^5\,|\, x_4=x_5=0\}.$ In Figure~\ref{fig:dae} we observe the subspace turnpike behavior proven in Theorem~\ref{thm:DAE_tp} for the corresponding DAE-OCP \eqref{e:phDAE_OCP}, i.e., the optimal state is close to $\ker(RQ)$ for the majority of the time.

\FloatBarrier
\section{Conclusion}
\label{s:conclusion}
This paper has investigated a class of optimal control problems for linear port-Hamil\-tonian descriptor systems. We have shown that, considering the supplied energy as the objective to be minimized, the optimal solutions transferring initial data to prescribed target sets exhibit the turnpike phenomenon. Specifically, we have presented results on input-state subspace turnpikes in the ODE-constrained reduction of the original DAE-constrained problem. We have shown that the  input-state subspace ODE turnpike corresponds to a state subspace DAE turnpike in the original problem. Importantly, we generalized the classical notion of steady-state turnpikes to subspace turnpikes. In the context of pH systems this turnpike subspace, which can be regarded as the attractor of infinite-horizon optimal solutions, is the nullspace of the dissipation matrix $RQ$. 
Future work will consider the extension towards infinite-dimensional systems (we refer to~\cite{Philipp2021} for first steps in this direction) and towards using the global dissipation inequality to derive global turnpike results for nonlinear port-Hamiltonian dynamics, e.g., irreversible systems occurring in thermodynamics.

\bibliographystyle{abbrv}
\bibliography{references.bib}

\
\\
\appendix
\section{A short solution theory for regular DAEs}\label{a:dae_solutions}
Let $b : [0,\infty)\to\R^n$ be measurable and $E,A\in\R^{n\times n}$. We consider the initial value problem (IVP)
\begin{equation}\label{e:ivp}
\tfrac d{dt}Ex = Ax + b,\quad Ex(0) = Ex^0.
\end{equation}
Here, we shall assume that the corresponding pencil $P(s) = sE-A$ is regular, i.e., there exists $\mu\in\C\backslash\{0\}$, such that $\mu E - A$ is invertible. We shall now present a subspace version of the quasi-Weierstra\ss\  form (cf.\ \cite{Berger2012}) which simplifies the solution analysis for \eqref{e:ivp}. Set $T \doteq   (\mu E - A)^{-1}E$ and consider the spaces\footnote{It can be shown that $V_{k+1} = A^{-1}EV_k$ (pre-image) and $W_{k+1} = E^{-1}W_k$, where $V_0 = \R^n$ and $W_0 = \{0\}$. These sequences of linear spaces are called {\em Wong sequences}.}
$$
V_k \doteq   \im T^k
\qquad\text{and}\qquad
W_k \doteq   \ker T^k.
$$
It is clear that $V_{k+1}\subset V_k$ and $W_k\subset W_{k+1}$. Moreover, it is easy to show that $V_{k+1} = V_k$ implies $V_{k+j} = V_k$ for all $j\in\N$. The smallest $k$ for which this happens is called the {\em descent} of $T$.  We denote it by $\delta(T)$. Similarly, $W_{k+1} = W_k$ implies $W_{k+j} = W_k$ for all $j\in\N$. The smallest such $k$ is called the {\em ascent} of $T$ which we denote by $\alpha(T)$. The following lemma is well-known.

\begin{lemma}\label{l:ascdesc}
We have $\alpha(T) = \delta(T) =: m$ and 
\begin{equation}\label{e:decVW}
\R^n = V_m\oplus W_m.
\end{equation}
\end{lemma}
\begin{proof}
The rank-nullity formula $n = \dim\ker T^k + \dim\im T^k$ implies $\alpha(T) = \delta(T)$. If $x\in V_m\cap W_m$, then $T^mx=0$ and $T^my=x$ for some $y$, thus $T^{2m}y=0$, which implies $x=T^my=0$. Hence, $V_m\cap W_m=\{0\}$ and \eqref{e:decVW} follows again from the rank-nullity formula.
\end{proof}

The number $m$ appearing in Lemma \ref{l:ascdesc} is called the {\em differentiation index} of the DAE in \eqref{e:ivp}. At the same time, it is called the {\em index} of the pencil $sE-A$.

Set $V \doteq   V_m$ and $W \doteq   W_m$. The spaces $V$ and $W$ are obviously $T$-invariant. Let $T_V : V\to V$ and $T_W : W\to W$ be the restrictions of $T$ to $V$ and $W$, respectively. Note that $T_W$ is nilpotent with nilpotency index $m$. Therefore, $\mu T_W - I_W$ is invertible. Since $\ker T\subset W$, also $T_V$ is invertible. Now, define the maps (on $W$, $V$, and $\R^n$, respectively)
$$
N \doteq   T_W(\mu T_W-I_W)^{-1}, \quad C \doteq   \mu I_V - T_V^{-1}, \quad\text{and }S \doteq   EP_V + AP_W,
$$
where $P_V$ ($P_W$) denotes the projection onto $V$ ($W$, resp.) with respect to the decomposition \eqref{e:decVW}. The map $S$ is invertible. Indeed, if $EP_Vx + AP_Wx=0$, we apply $(\mu E - A)^{-1}$ and obtain $T_VP_Vx + (\mu T_W - I_W)P_Wx = 0$. This implies $P_Vx = 0$ and $P_Wx = 0$, thus $x=0$. From $AT = E(\mu T-I)$ it is now easily seen that
\begin{align}\label{e:quasi-weierstrass}
S(C\oplus I_W) = A\qquad\text{and}\qquad S(I_V\oplus N) = E.
\end{align}
Therefore, the IVP \eqref{e:ivp} transforms into
$$
\tfrac d{dt}(I_V\oplus N)x = (C\oplus I_W)x + S^{-1}b, \qquad (I_V\oplus N)x(0) = (I_V\oplus N)x^0.
$$
Setting $c \doteq   -S^{-1}b$, we equivalently get
\begin{subequations}\label{e:ivp_sys}
\begin{alignat}{3}
\dot x_V &= Cx_V - c_V,\qquad &&\phantom{N}x_V(0) &&= x_V^0,\label{e:ivp_sys1}\\
\tfrac d{dt}Nx_W &= x_W - c_W,&&Nx_W(0) &&= Nx_W^0.\label{e:ivp_sys2}
\end{alignat}
\end{subequations}
Obviously, \eqref{e:ivp_sys1} is an ODE IVP and thus has a unique solution. In what follows, we set $W^{1,1} \doteq   W^{1,1}_{\rm loc}([0,\infty))$.
\newpage
\begin{proposition}\label{p:vks}
The following hold:
\begin{enumerate}
\item[{\rm (i)}] If $m=0$, then $E$ is invertible and the IVP \eqref{e:ivp} has a unique solution.
\item[{\rm (ii)}] If $m=1$, then $N=0$ and \eqref{e:ivp_sys2} has the unique solution $x_W = c_W$. Hence, the original IVP \eqref{e:ivp} has a unique solution.
\item[{\rm (iii)}] If $m\ge 2$, then the DAE in \eqref{e:ivp_sys2} has a solution if and only if $v_{1} \doteq   N^{m-1}c_W\in W^{1,1}$ and
$$
v_{k} \doteq   N^{m-k}c_W + \dot v_{k-1}\in W^{1,1},\quad k=2,\ldots,m-1.
$$
The solution is unique, equals $x_W = c_W + \dot v_{m-1}$, and satisfies $Nx_W = v_{m-1}$. Hence, the IVP \eqref{e:ivp} has a \braces{unique} solution if and only if $v_1,\ldots,v_{m-1}\in W^{1,1}$ and $v_{m-1}(0) = Nx_W^0$.
\end{enumerate}
\end{proposition}
\begin{proof}
The assertions (i) and (ii) are immediate. So, let $m\ge 2$. Assume $v_1,\ldots,v_{m-1}\in W^{1,1}$. We have $Nv_{1} = 0$ and thus $Nv_{2} = N^{m-1}c_W + N\dot v_{1} = v_{1}$. Also, $Nv_3 = N^{m-2}c_W + N\dot v_2 = v_2$ and so on until we arrive at $Nv_{m-1} = v_{m-2}$. Therefore, if we set $x_W \doteq   c_W + \dot v_{m-1}$,
$$
\tfrac d{dt}Nx_W = \tfrac d{dt}N(c_W + \dot v_{m-1}) = \tfrac d{dt}(Nc_W + \dot v_{m-2}) = \dot v_{m-1} = x_W-c_W.
$$
To see that the solution is unique, let $u$ and $v$ be two solutions. Then $w = u-v$ satisfies $\frac d{dt}Nw = w$. Applying $N^{m-1}$ to this equation gives $N^{m-1}w = 0$. Thus, applying $N^{m-2}$ to the equation yields $N^{m-2}w = 0$ and so on, so that finally we obtain $w=0$.

Conversely, let $x_W$ be a solution of the DAE. Again, if we subsequently apply $N^k$, $k=m-1,\ldots,1$, to the DAE, we see that  $v_1 = N^{m-1}c_W = N^{m-1}x_W\in W^{1,1}$ and $v_k = N^{m-k}c_W + \dot v_{k-1} = N^{m-k}x_W\in W^{1,1}$ for $k=2,\ldots,m-1$.
\end{proof}

\section{Dissipative Hamiltonian matrices and pencils}\label{a:matpencils}
In this section of the Appendix we analyze and characterize the class of the pencils $P(s) = sE-(J-R)Q$ which correspond to the DAEs in the port-Hamiltonian systems we consider in this paper. The section is separated into two parts, where the second builds upon the first in which we restrict the analysis to the subclass where $E=I$.

\subsection{Dissipative Hamiltonian matrices}\label{a:diss-ham-mat}
In what follows we analyze and characterize the class of matrices of the form $A = (J-R)Q$ with $J,R,Q\in\R^{n\times n}$ s.t.
\begin{equation}\label{e:phODE_matrix_conditions2}
	J = -J^\top, \quad R=R^\top\ge0,\quad Q=Q^\top\ge 0.
\end{equation}

\begin{definition}
	We say that a matrix $A\in\R^{n\times n}$ is {\em dissipative Hamiltonian} if it admits a representation $A = (J-R)Q$, where $J,R,Q\in\R^{n\times n}$ are as in \eqref{e:phODE_matrix_conditions2}. If there exists such a representation of $A$ in which $Q$ is positive definite, we say that the dissipative Hamiltonian matrix $A$ is {\em non-degenerate}.
\end{definition}
Our first observation is that the set of dissipative Hamiltonian matrices is invariant under similarity transforms. Indeed, for any invertible $S\in\R^{n\times n}$ we have
$$
S\big[(J-R)Q\big]S^{-1} = (SJS^\top-SRS^\top)(S^{-\top}QS^{-1}).
$$
The next result shows that being dissipative Hamiltonian is a pure spectral property.
\newpage
\begin{theorem}\label{t:ph-charac}
	A matrix $A\in\R^{n\times n}$ is dissipative Hamiltonian if and only if all of the following conditions are satisfied:
	\begin{enumerate}[{\rm (i)}]
		\item $\sigma(A)\subset\{\la\in\C : \Re\la\le 0\}$.
		\item $\ker\big((A-i\alpha)^2\big) = \ker(A-i\alpha)$ for $\alpha\in\R\backslash\{0\}$.
		\item $\ker A^3 = \ker A^2$.
	\end{enumerate}
	If $A = (J-R)Q$, then for $\alpha\in\R\backslash\{0\}$ we have
	\begin{equation}\label{e:ialpha}
		\ker\big((A-i\alpha)^2\big) = \ker(A-i\alpha) = \ker(JQ-i\alpha)\cap\ker(RQ).
	\end{equation}
	Moreover, $\ker A = \ker(JQ)\cap\ker(RQ)$ and
	\begin{equation}\label{e:atzero}
		\ker A^3 = \ker A^2 = \ker(QJQ)\cap\ker(RQ).
	\end{equation}
\end{theorem}
\begin{proof}
	The necessity of (i)--(iii) is shown in the proof of the more general Theorem \ref{t:ph_pencil_charac} below. Let us prove the formulas for $\ker(A-i\alpha)$ and $\ker A^2$ if $A = (J-R)Q$. To this end, let $\alpha\in\R$ and $x\in\C^n$ such that $Ax = i\alpha x$. Then $\<RQx,Qx\> = -\Re\<(J-R)Qx,Qx\> = 0$ implies $RQx=0$ and thus also $JQx = i\alpha x$. This establishes \eqref{e:ialpha} and $\ker A = \ker(JQ)\cap\ker(RQ)$. To prove \eqref{e:atzero}, let $A^2y = 0$ for some $y\in\R^n$ and set $x \doteq   Ay$. Then $x\in\ker A$ and thus $\<Qx,x\> = \<Qx,(J-R)Qy\> = -\<(J+R)Qx,Qy\> = 0$, so $Qx=0$. In particular, $Q(J-R)Qy = Qx = 0$. Again, $-\<RQy,Qy\> = \Re\<(J-R)Qy,Qy\> = 0$, which implies $RQy = 0$ and hence also $QJQy = 0$.
	
	For the sufficiency part, assume that (i)--(iii) hold. Assume first that $\ker A^2 = \ker A$. Then \cite[Corollary III.1]{Carlson63} implies the existence of a Hermitian positive definite matrix $H\in\C^{n\times n}$ such that $-AH - HA^\top\ge 0$. Note that the complex conjugate of a positive (semi-)definite matrix remains positive (semi-)definite. Hence, $-A\ol H - \ol HA^\top\ge 0$, and thus $R \doteq   -\frac 12(AP + PA^\top)\ge 0$, where $P = H+\ol H\in\R^{n\times n}$ is positive definite. Thus, setting $J \doteq   \frac 12(AP-PA^\top)$ and $Q \doteq   P^{-1}$ gives
	$$
	A = \big(\tfrac 12(AP-PA^\top) + \tfrac 12(AP+PA^\top)\big)P^{-1} = (J-R)Q.
	$$
	In the case $\ker A^2\neq\ker A$, there exists a non-singular matrix $S\in\R^{n\times n}$ such that $SAS^{-1} = A_1\oplus A_2$ (real Jordan form), where $A_1$ is of the form
	$$
	A_1 = \bigoplus_{i=1}^k\mat 0 1 0 0
	$$
	and $A_2$ is real and enjoys the properties (i)--(ii) of $A$ and $\ker A_2^2 = \ker A_2$. By the above reasoning, $A_2$ is dissipative Hamiltonian. Now, as
	$$
	\mat 0 1 0 0 = \Bigg[\underbrace{\mat 0 1 {-1} 0}_{=J} - \underbrace{\mat 0 0 0 0}_{=R}\Bigg]\underbrace{\mat 0 0 0 1}_{=Q},
	$$
	is dissipative Hamiltonian, the same follows for $A_1\oplus A_2 = SAS^{-1}$ and hence for $A$.
\end{proof}

\begin{remark}
	The proof of Theorem \ref{t:ph-charac} shows that in the non-degenerate case the matrix $Q$ can be derived as the inverse of a positive definite solution of the linear matrix inequality $AX+XA^\top\le 0$.
\end{remark}

The following corollaries directly follow from Theorem \ref{t:ph-charac} and its proof.

\begin{corollary}\label{c:non-deg1}
	A matrix $A\in\R^{n\times n}$ is non-degenerate dissipative Hamiltonian if and only if it satisfies {\rm (i)} and {\rm (ii)} in Theorem \rmref{t:ph-charac} and $\ker A^2 = \ker A$.
\end{corollary}

\begin{corollary}
	If $A$ is dissipative Hamiltonian, then so is $A^\top$. If, in addition, $A$ is invertible, then also $A^{-1}$ is dissipative Hamiltonian.
\end{corollary}

\subsection{Dissipative Hamiltonian pencils}\label{a:diss-ham-pencil}
Next, we extend our analysis of dissipative Hamiltonian matrices in Appendix \ref{a:diss-ham-mat} to matrix pencils.

\begin{definition}\label{d:diss_ham}
	We say that a matrix pencil $P(s) = sE-A$ with $E,A\in\R^{n\times n}$ is {\em dissipative Hamiltonian} if $A = (J-R)Q$ with matrices $J,R,Q\in\R^{n\times n}$ satisfying
	\begin{equation}\label{e:matrix_conditions2}
		J=-J^\top,\quad R=R^\top\geq 0,\quad Q^\top E = E^\top Q\geq 0.
	\end{equation}
	A dissipative Hamiltonian pencil is called {\em non-degenerate} if $Q$ can be chosen invertible.
\end{definition}

Obviously, a real square matrix $A$ is dissipative Hamiltonian if and only if the pencil $sI-A$ is dissipative Hamiltonian. The class of dissipative Hamiltonian pencils is invariant under multiplication with invertible matrices from the left or from the right. Indeed, if $U,V\in\R^{n\times n}$ are invertible, then
$$
U\big[sE - (J-R)Q\big]V = \la (UEV) - \big(UJU^\top - URU^\top\big)\big(U^{-\top}QV\big)
$$
and $(UEV)^\top(U^{-\top}QV) = V^\top E^\top QV\ge 0$.

A number $\la\in\C$ is called an {\em eigenvalue} of the pencil $P(s) = sE-A$ if there exists a vector ({\em eigenvector}) $x\neq 0$ such that $P(\la)x=0$. The point $\la=\infty$ is an eigenvalue of $P$ if there exists $x\neq 0$ such that $Ex=0$. If $\la\in\C\cup\{\infty\}$ is an eigenvalue of $P$, then a tuple $(x_0,x_1,\ldots,x_k)\in(\C\backslash\{0\})^{(k+1)n}$ is called a {\em chain} of $P$ at $\la$ of length $k+1$ if
\begin{align*}
	\la\in\C\,&:\quad(A-\la E)x_0 = 0,\quad (A-\la E)x_1 = Ex_0,\;\ldots\;,(A-\la E)x_k = Ex_{k-1},\\
	\la = \infty\,&: \quad Ex_0=0,\quad Ex_1 = Ax_0,\;\ldots\,,Ex_k = Ax_{k-1}.
\end{align*}
An eigenvalue of $P$ is called {\em semi-simple}, if there are no chains of length greater than one associated to it. The {\em index} of the pencil $P$ is defined as the maximal length of chains associated to $\la = \infty$. 
The next theorem is the main result of this subsection.

\begin{theorem}\label{t:ph_pencil_charac}
	Let $E,A\in\R^{n\times n}$, $P(s) = sE-A$, and $P'(s) = sA-E$. Assume that $P$ is regular. Then $P$ is dissipative Hamiltonian if and only if it has all of the following spectral properties:
	\begin{enumerate}[{\rm (i)}]
		\item $\Re\la\le 0$ for each eigenvalue $\la\in\C$ of $P$.
		\item Non-zero imaginary eigenvalues of $P$ are semi-simple.
		\item The index of $P$ is at most two.
		\item The index of $P'$ is at most two.
	\end{enumerate}
\end{theorem}
\begin{proof}
	The necessity of (i)--(iii) has already been proved in \cite{Mehl2018} (independently of whether $P$ is regular or not). The fact that also (iv) is necessary follows from \cite[Theorem 6.1]{Mehl2018} and the proof of \cite[Corollary 6.2]{Mehl2018}. 
	For the sufficiency part, let (i)--(iv) be satisfied. By \cite{Berger2012} there exist invertible matrices $U,V\in\R^{n\times n}$ such that
	$$
	UP(s)V = \mat{sI - A'}00{sN - I},
	$$
	where $A'$ and $N$ are square matrices with $N$ being nilpotent.
	By (i), (ii), and (iv), the matrix $A'$ enjoys the properties (i)--(iii) in Theorem \ref{t:ph-charac} and is therefore dissipative Hamiltonian. Hence, the pencil $sI-A'$ is dissipative Hamiltonian. Furthermore, (iii) implies that the nilpotency index of $N$ is at most $2$, that is, $N^2=0$. Therefore, $N$ is similar to ${\bf 0}\oplus N_1$, where
	$$
	N_1 = \bigoplus_{i=1}^k\mat 0100.
	$$
	Obviously, $s\cdot {\bf 0} - I = s\cdot{\bf 0} - ({\bf 0} - I)(-I)$ is dissipative Hamiltonian. Furthermore, we may write
	$$
	s\mat 0100 - \mat 1001 = s\underbrace{\mat 0100}_{E_i} - \Bigg[\underbrace{\mat 0{-1}10}_{J_i} - \underbrace{\mat 0000}_{R_i}\Bigg]\underbrace{\mat 01{-1}0}_{Q_i},
	$$
	where $E_i^\top Q_i = Q_i^\top E_i\ge 0$. This proves that $sE-A$ is dissipative Hamiltonian.
\end{proof}

\begin{corollary}
	If the regular pencil $P$ is dissipative Hamiltonian, then so is $P'$.
\end{corollary}

Our next goal is to find a characterization for the regularity of dissipative Hamiltonian pencils $P(s) = sE - (J-R)Q$. For this, let us first study the connection between $E$ and $Q$. Note that $E^\topp Q = Q^\topp E$ implies
\begin{equation}\label{e:kerEQ}
	E\ker Q\subset\ker Q^\topp 
	\qquad\text{and}\qquad
	Q\ker E\subset\ker E^\topp .
\end{equation}
Hence, for every $s\in\C$ we have
\begin{equation}\label{e:kerPs}
	P(s)\ker Q\,\subset\,\ker Q^\topp 
	\qquad\text{and}\qquad
	P(s)\ker E\,\subset\,(J-R)\ker E^\topp .
\end{equation}
\begin{lemma}\label{l:QE}
	If $\ker E\cap\ker Q = \{0\}$, then the following statements hold:
	\begin{enumerate}
		\item[{\rm (i)}]   $\ker(E^\topp Q) = \ker E\oplus\ker Q$.
		\item[{\rm (ii)}]  Both $E$ and $P(s)$ map $\ker Q$ bijectively onto $\ker Q^\topp $, $s\neq 0$.
		\item[{\rm (iii)}] We have $\ker Q^\topp  = E\ker Q = \im E\cap\ker Q^\topp$ and in particular, $\ker Q^\topp \subset\im E$.
	\end{enumerate}
\end{lemma}
\begin{proof}
	Clearly, $\ker E\subset\ker(Q^\topp E)$ and $\ker Q\subset\ker(E^\topp Q) = \ker(Q^\topp E)$. Hence, $\ker E\oplus\ker Q\subset\ker(Q^\topp E).$ Now, consider the linear map $T : \ker(Q^\topp E)\to\ker Q^\topp \cap\im E$, defined by $Tx = Ex$, $x\in\ker(Q^\topp E)$. It is easily seen that $\ker T = \ker E$ and $\im T = \ker Q^\topp \cap\im E$. Hence, (i) follows from
	\begin{align*}
		\dim\ker(Q^\topp E)
		&= \dim\ker T + \dim\im T = \dim\ker E + \dim\!\big(\!\ker Q^\topp \cap\im E)\\
		&\le \dim\ker E + \dim\ker Q^\topp = \dim\ker E + \dim\ker Q.
	\end{align*}
	From \eqref{e:kerEQ} and \eqref{e:kerPs} we see that $E$ and $P(s) = sE - (J-R)Q$ indeed map $\ker Q$ into $\ker Q^\topp $. Moreover, if $Ex=0$ or $P(s)x = 0$, $x\in\ker Q$, then $sEx=0$, hence $x\in\ker E\cap\ker Q = \{0\}$. So, $E|_{\ker Q}$ and $P(s)|_{\ker Q}$ are injective and the claim follows from $\dim\ker Q = \dim\ker Q^\topp $. Finally, (iii) follows immediately from (ii).
\end{proof}

\begin{proposition}\label{p:regular}
	The following are equivalent:
	\begin{enumerate}[{\rm (a)}]
		\item The dissipative Hamiltonian pencil $P(s) = sE - (J-R)Q$ is regular.
		\item $\ker(Q^\topp JQ)\,\cap\,\ker(RQ)\,\cap\,\ker(E) = \{0\}$.
		\item $\ker(Q^\topp JQ)\,\cap\,\ker(RQ)\,\cap\,\ker(E^\top Q) = \ker Q$ and $\ker E\cap\ker Q = \{0\}$.
	\end{enumerate}
\end{proposition}
\begin{proof}
	(a)$\Sra$(b). Let $P$ be regular. Then there is $s > 0$ such that $P(s)$ is invertible. In particular, $\ker E\cap\ker Q = \{0\}$. If $x\in\ker(Q^\topp JQ)\cap\ker(RQ)\cap\ker(E)$, then $P(s)x = -JQx\in\ker Q^\topp $. By Lemma \ref{l:QE} (ii), we find $v\in\ker Q$ such that $P(s)v = P(s)x$, which implies $x = v\in\ker Q$. But then $P(s)x = -JQx = 0$, hence $x=0$.
	
	(b)$\Sra$(c). Let $Q^\top JQx = RQx = E^\top Qx = 0$. By Lemma \ref{l:QE} (i), we may write $x = u+v$ with $u\in\ker E$ and $v\in\ker Q$. From $Qx = Qu$ we conclude that $Q^\topp JQu = RQu = Eu = 0$. Hence, $u=0$ by (b) and thus $Qx=0$.
	
	(c)$\Sra$(a). Let $x\in\C^n$ such that $Ex = (J-R)Qx$. Then $0\le\<Q^\topp Ex,x\> = \<Ex,Qx\> = -\<RQx,Qx\>\le 0$ implies $Q^\topp Ex = RQx = 0$ and thus also $Q^\top JQx=0$. Hence, $Qx=0$ and therefore also $Ex=0$. It follows that $x=0$.
\end{proof}

\begin{corollary}
	\label{c:ind1}
	A regular pencil $P$ is non-degenerate dissipative Hamiltonian if and only if {\rm (i)--(iii)} hold in Theorem \rmref{t:ph_pencil_charac} and $P'$ has index at most one.
\end{corollary}
\begin{proof}
	Assume that the index of $P'$ is at most one. Then the matrix $A'$ in the proof of Theorem \ref{t:ph_pencil_charac} satisfies $\ker A' = \ker(A')^2$ and is therefore a non-degenerate dissipative Hamiltonian matrix by Corollary \ref{c:non-deg1}. Furthermore, since the matrices $-I$ and $Q_i$ in the proof of Theorem \ref{t:ph_pencil_charac} are invertible, it follows that the pencil $sN-I$ is non-degenerate. Hence, so is $P$. Conversely, assume that $P(s) = sE - (J-R)Q$ with $Q$ invertible and consider a chain of length two at $\la=0$, i.e., $(J-R)Qx_0 = 0$ and $(J-R)Qx_1=Ex_0$. Then $RQx_0 = JQx_0 = 0$ and also $E^\top Qx_0 = 0$ since $\<Qx_0,Ex_0\> = \<Qx_0,(J-R)Qx_1\> = -\<(J+R)Qx_0,Qx_1\> = 0$. Hence, $Qx_0 = 0$ by Proposition \ref{p:regular} (c) and thus $x_0 = 0$, which shows that the index of $P'$ is at most one.
\end{proof}

\begin{proposition}\label{p:ind1}
The dissipative Hamiltonian pencil $P(s) = sE - (J-R)Q$ is regular with index at most one if and only if 
\begin{align}
\label{e:reg_ind1}
\ker E\,\cap\,\ker(RQ)\,\cap\,(JQ)^{-1}\im E = \{0\}-
\end{align}
\end{proposition}
\begin{proof}
Note that the vector $x_0$ of a chain $(x_0,x_1)$ at $\la=\infty$ of a pencil $sE-A$ belongs to $\ker E\cap A^{-1}\im E$. Conversely, if $x_0\in\ker E\cap A^{-1}\im E$, then there exists $x_1$ such that $Ex_1 = Ax_0$. Hence, the pencil's index is at most one if and only if $\ker E\cap A^{-1}\im E = \{0\}$.

First of all we shall show that (in any case)
$$
\ker E\,\cap\,[(J-R)Q]^{-1}\im E = \ker E\,\cap\,\ker(RQ)\,\cap\,(JQ)^{-1}\im E.
$$
Indeed, if $Ex=0$ and $(J-R)Qx = Ev$ for some $v\in\R^n$, then $-\<RQx,Qx\> = \<(J-R)Qx,Qx\> = \<Ev,Qx\> = \<v,E^\topp Qx\> = \<v,Q^\topp Ex\> = 0$ and thus $RQx=0$, so that $x$ is contained in the set on the right-hand side. The converse inclusion is trivial.

Hence, if the pencil is regular with index $1$, then \eqref{e:reg_ind1} holds. Conversely, assume that \eqref{e:reg_ind1} holds. Then $\ker E\cap\ker Q = \{0\}$ and due to Lemma \ref{l:QE} we have
$$
\ker(Q^\topp JQ) = (JQ)^{-1}\ker Q^\topp \subset (JQ)^{-1}\im E,
$$
and thus Proposition \ref{p:regular} implies that $P$ is regular.
\end{proof}

\section{Existence of optimal solutions}\label{a:existence}
Let $A\in\R^{n\times n}$, $B\in\R^{n\times m}$, and $W\in\R^{k\times k}$, where $k = n+m$. Moreover, let $\UU\subset\R^m$, $\Phi\subset\R^n$, $\calI = [0,T]$, $x^0\in\R^n$, and consider the optimal control problem
\begin{align}
	\begin{split}\label{e:FC_OCPlin}
		\min_{u\in L^2(\calI,\UU)}\,&M(x(T)) + \int_0^T\left\|W\smallvek{x(t)}{u(t)}\right\|^2\,dt\\
		\text{s.t.}\quad &\dot x = Ax + Bu,\\
		&x(0)=x^0,\quad x(T)\in\Phi.
	\end{split}
\end{align}

The next theorem shows that under suitable conditions on the problem data, the OCP \eqref{e:FC_OCPlin} is well-posed if only $\Phi$ is reachable from $x^0$ at time $T$.

\begin{theorem}\label{t:optsol_exists}
	Assume the following:
	\begin{itemize}
		\item $\UU$ is convex and compact
		\item $\Phi$ is closed.
		\item $M : \R^n\to\R$ is convex \braces{and hence continuous}.
	\end{itemize}
	If $x^0\in\RT_T(\Phi)$, then there exists an optimal control for the OCP \eqref{e:FC_OCPlin}.
\end{theorem}

In the proof we will twice make use of the following simple consequence of the Hahn-Banach separation theorem.

\begin{lemma}\label{l:clcv}
	A closed and convex subset $C$ of a normed space $X$ is weakly sequentially closed. That is, if $(x_n)\subset C$ and $x_n\wto x\in X$ \braces{weak convergence}, then $x\in C$.
\end{lemma}

\begin{proof}[Proof of Theorem \ref{t:optsol_exists}]
	Denote by $\calU$ the set of admissible controls for \eqref{e:FC_OCPlin}. That is,
	\begin{align*}
		\calU
		&= \left\{u\in L^2(\calI,\UU) : e^{TA}x^0 + \int_0^Te^{(T-t)A}Bu(t)\,dt\in\Phi\right\}.
	\end{align*}
	With the linear and bounded finite rank operator $\vphi : L^2(\calI,\R^m)\to\R^n$, defined by $\vphi(u) \doteq   \int_0^Te^{(T-t)A}Bu(t)\,dt$, we can write this as
	$$
	\calU = \vphi^{-1}\left(\left\{\Phi-e^{TA}x^0\right\}\right)\cap L^2(\calI,\UU).
	$$
	The set $\calU$ is closed. Indeed, $\vphi^{-1}(\{\Phi-e^{TA}x^0\})$ is closed since $\Phi$ is closed and $\vphi$ is continuous, and if $(u_n)\subset L^2(\calI,\UU)$ converges in $L^2(\calI,\R^m)$ to some $u\in L^2(\calI,\R^m)$, then by a theorem of Riesz there exists a subsequence $(u_{n_k})$ such that $u_{n_k}(t)\to u(t)$ for a.e.\ $t\in\calI$ and thus $u\in L^2(\calI,\UU)$. Since $L^2(\calI,\UU)$ is bounded, the same holds for $\calU$.
	
	Define the (non-linear) functionals $J_1,J_2 : L^2(\calI,\R^m)\to\R$ by
	$$
	J_1(u) \doteq   M\left(e^{TA}x^0 + \vphi(u)\right),\quad
	J_2(u) \doteq   \int_0^T\left\|W\smallvek{x(t,x^0;u)}{u(t)}\right\|^2\,dt,\quad u\in L^2(\calI,\R^m).
	$$
	Due to the assumptions on $M$, the functional $J_1$ is convex and continuous. Since $x(\,\cdot\,,x^0;u)$ depends affine-linearly on $u$, we may write $J_2$ as $J_2(u) = \int_0^T\|(Cu)(s)\|^2\,ds$ with a continuous affine-linear operator $C : L^2(\calI,\R^m)\to L^2(\calI,\R^k)$. From this representation and the convexity of the function $z\mapsto\|z\|^2$ on $\R^k$ we see that $J_2$ is convex. Furthermore, $J_2$ is continuous, even Lipschitz on bounded subsets. Overall, this shows that the cost functional $J \doteq   J_1 + J_2$ is convex and continuous.
	
	Let $a \doteq   \inf\{J(u) : u\in\calU\}$. Then there exists a sequence $(u_n)\subset\calU$ such that $J(u_n)\to a$ as $n\to\infty$. As $\calU$ is bounded, we may assume that $(u_n)$ converges weakly to some $u^\star \in L^2(\calI,\R^m)$. By Lemma \ref{l:clcv} we have $u^\star \in L^2(\calI,\UU)$. Moreover, since $\vphi$ is linear and finite-rank, from $\vphi(u_n)\in\Phi - e^{TA}x^0$ we conclude that also $\vphi(u^\star ) = \lim_{n\to\infty}\vphi(u_n)\in\Phi - e^{TA}x^0$. Hence, $u^\star \in\calU$.
	
	For $\veps > 0$ set $L_\veps \doteq   \{u\in L^2(\calI,\R^m) : J(u)\le a+\veps\}$. Convexity and continuity of $J$ imply that each $L_\veps$ is convex and closed. Since $u_n\in L_\veps$ for infinitely many $n\in\N$, it is another consequence of Lemma \ref{l:clcv} that $u^\star \in L_\veps$ and thus $J(u^\star )\le a+\veps$. As $\veps > 0$ was arbitrary, this implies $J(u^\star )\le a$. On the other hand, $u^\star $ is admissible and hence $J(u^\star )\ge a$ by definition of $a$. Thus, $J(u^\star ) = a$.
\end{proof}

\end{document}